\documentclass[11pt]{amsart}

\usepackage{amsmath, amsthm, amssymb, graphicx}

\newcounter{citedtheorems}

\newtheorem{defn}{Definition}[section]
\newtheorem{theorem}[defn]{Theorem}
\newtheorem*{theorem-abs}{Theorem \ref{m1}}
\newtheorem{thm-lit}[citedtheorems]{Theorem}
\newtheorem{defn-lit}[citedtheorems]{Definition}
\newtheorem{fact}[defn]{Fact}
\newtheorem{cor}[defn]{Corollary}

\newtheorem{concl}[defn]{Conclusion}

\newtheorem{claim}[defn]{Claim}
\newtheorem{lemma}[defn]{Lemma}
\newtheorem{obs}[defn]{Observation}
\newtheorem{rmk}[defn]{Remark}

\newtheorem{ntn}[defn]{Notation}
\newtheorem{disc}[defn]{Discussion}

\newtheorem{qst}[defn]{Question}

\newtheorem{hyp}[defn]{Hypothesis}

\newcommand{\br}{\vspace{2mm}}
\newcommand{\step}{\noindent\emph}

\newcommand{\vrt}{\rule{0pt}{12pt}}

\newcommand{\mcu}{\mathcal{U}}

\newcommand{\eff}{\mathcal{F}}

\newcommand{\si}{{i}}
\newcommand{\trv}{\mathbf{t}} 
\newcommand{\uu}{\mathcal{U}}

\newcommand{\uw}{\mathcal{W}}
\newcommand{\nn}{\mathbf{n}}
\newcommand{\mem}{\operatorname{mem}}
\newcommand{\fb}{\textbf{f}}

\newcommand{\vp}{\varphi}

\newcommand{\emp}{\mathcal{P}}

\title{Regularity lemmas for stable graphs}

\author{M. Malliaris and S. Shelah}\thanks{Shelah would like to thank the Israel
Science Foundation for partial support of this research via grant 710/07. 
\\ Malliaris would like to thank the NSF which partially supported this research, 
as well as Malliaris' visit to Rutgers in September 2010,
via grant DMS-1001666 and Shelah's grant DMS-0600940.}
\address{Department of Mathematics, University of Chicago, 5734 S. University Avenue, Chicago, IL 60637, USA}
\email{mem@math.uchicago.edu}

\address{Einstein Institute of Mathematics, Edmond J. Safra Campus, Givat Ram, The Hebrew
University of Jerusalem, Jerusalem, 91904, Israel, and Department of Mathematics,
Hill Center - Busch Campus, Rutgers, The State University of New Jersey, 110
Frelinghuysen Road, Piscataway, NJ 08854-8019 USA}
\email{shelah@math.huji.ac.il}
\urladdr{http://shelah.logic.at}

\begin{document}

\begin{abstract}
Let $G$ be a finite graph with the non-$k_*$-order property (essentially, a uniform finite bound on the size of an induced
sub-half-graph). A major result of the paper applies model-theoretic arguments 
to obtain a stronger version of Szemer\'edi's regularity lemma for such graphs, 
in which there are \emph{no} irregular pairs, the bounds are significantly improved, 
and each component satisfies an indivisibility condition:

\begin{theorem-abs}  
Let $k_* \in \mathbb{N}$ and therefore $k_{**}$ \emph{(}a constant depending on $k_*$, but $\leq 2^{k_*+2}$\emph{)} be given. 
Let $G$ be a finite graph with the non-$k_*$-order property. 
Then for any $\epsilon > 0$ there exists $m=m(\epsilon)$ such that for all sufficiently large $A \subseteq G$, 
there is a partition $\langle A_i : i<i(*)\leq m\rangle$ of $A$ into at most $m$ pieces, where:
\begin{enumerate}
 \item for all $i,j<i(*)$, $||A_i|-|A_j||\leq 1$
 \item \emph{all} of the pairs $(A_i, A_j)$ are $(\epsilon, \epsilon)$-uniform, 
\\meaning that for some truth value $\trv=\trv(A_i,A_j) \in \{0,1\}$,
for all but $<\epsilon|A_i|$ of the elements of $|A_i|$, for all but $< \zeta|A_j|$ of the elements of $A_j$,  $(aRb) \equiv \trv(A_i,A_j)$
 \item all of the pieces $A_i$ are $\epsilon$-excellent \emph{(}an indivisibility condition, Definition \ref{good} below\emph{)}
 \item if $\epsilon < \frac{1}{2^{k_{**}}}$, then $m \leq (3+\epsilon)\left(\frac{8}{\epsilon}\right)^{k_{**}}$
\end{enumerate}
\end{theorem-abs}

Motivation for this work comes from a coincidence of model-theoretic and graph-theoretic ideas. 
Namely, it was known that the ``irregular pairs'' in the statement of Szemer\'edi's regularity lemma
cannot be eliminated, due to the counterexample of half-graphs. The results of this paper show in what sense this counterexample is the 
only essential difficulty. The proof is largely model-theoretic (though written to be accessible to finite combinatorialists): 
arbitrarily large half-graphs coincide with model-theoretic instability,
so in their absence, structure theorems and technology from stability theory apply. In addition to the theorem quoted, we give
several other regularity lemmas with different advantages, in which the indivisibility condition on the components is improved 
(at the expense of letting the number of components grow with $|G|$) 
and extend some of these results to the larger class of graphs without the independence property. 
\end{abstract}

\maketitle

\section{Introduction}

This paper applies ideas from model theory to give stronger regularity lemmas for certain natural classes of graphs. We first state Szemer\'edi's celebrated regularity lemma. (The reader is also referred to the excellent survey \cite{ks}.)

Recall that if $A,B$ are finite graphs with disjoint vertex sets, the density $d(A,B) = \frac{|R\cap(A\times B)|}{|A||B|}$ and we say that
$(A,B)$ is $\epsilon$-regular if for all $A^\prime \subseteq A, B^\prime \subseteq B$ with $|A^\prime| \geq \epsilon |A|$, $|B^\prime| \geq 
\epsilon |B|$, we have that $|d(A,B) - d(A^\prime, B^\prime)| < \epsilon$.

\begin{thm-lit} \emph{(Szemer\'edi's regularity lemma)}
For every $\epsilon, m$ there exist $N = N(\epsilon, m)$, $m^\prime = m(\epsilon, m)$ such that given any finite graph $X$,
of size at least $N$, there is $k$ with $m\leq k \leq m^\prime$ and a partition $X = X_1 \cup \dots \cup X_k$ satisfying:
\begin{enumerate}
\item $||X_i| - |X_j|| \leq 1$ for all $i,j \leq k$
\item all but at most $\epsilon k^2$ of the pairs $(X_i, X_j)$ are $\epsilon$-regular. 
\end{enumerate}
\end{thm-lit}

As explained in \S 1.8 of \cite{ks}, ``Are there exceptional pairs?'' it was not known for some time whether the $\epsilon k^2$ irregular pairs allowed in clause (b) were necessary. Several researchers (Lovasz, Seymour, Trotter, as well as
Alon, Duke, Leffman, R\"odl, and Yuster in \cite{alon}) then independently observed that the \emph{half-graph}, i.e. the bipartite graph with vertex sets 
$\{ a_i : i<n \} \cup \{ b_i : i<n \}$ (for arbitrarily large $n$) such that $a_i R b_j$ iff $i<j$, shows that exceptional pairs are necessary.

It is therefore natural to ask whether ``half-graphs'' are the main difficulty, i.e.:

\begin{qst}
Consider the class of graphs which admit a uniform finite bound on the size of an induced sub-half-graph. 
It is possible to give a stronger regularity lemma for such graphs in which there are no irregular pairs?
\end{qst}

A major result of this paper is an affirmative answer to this question, Theorem \ref{m1} below, which both eliminates irregular pairs and also significantly improves the tower-of-exponential bounds of the Szemer\'edi lemma, which are necessary by work of Gowers \cite{gowers}. 
The point of entry to this proof is that, as model theorists will recognize, the half-graph is an instance of the \emph{order property}:

\begin{defn}
A formula $\vp(x_1,\dots x_\ell;y_1,\dots y_r)$  has the \emph{order property} with respect to some background theory $T$ if there exist, in 
some sufficiently saturated model $M$ of $T$, elements $\{ a^i_1,\dots a^i_\ell : i<\omega \}$ and $\{ b^j_1,\dots b^j_r : i<\omega \}$ such that
$\models \vp(a^i_1,\dots a^i_\ell ; b^j_1,\dots b^j_r) $ if and only if $i<j$.  
\end{defn}

Theories in which no formula has the order property are called \emph{stable}. Such theories have been fundamental to model theory
since the second author's work in \cite{Sh:c}, see e.g. the ``Unstable Formula Theorem'' II.2.2, p. 30. Generally speaking, one contribution of such model-theoretic analysis is to characterize global structural properties, such as number of models, existence of indiscernible sets, number of types, and so on, in terms of local combinatorial properties, such as the order property in some formula. 
By compactness, a formula has the order property (with respect to a background theory $T$) if and only if it has the $k$-order property for every natural number $k$, Definition \ref{non-op} below.  Note, for instance, that in a graph with the non-$k$-order property the density between sufficiently large $\epsilon$-regular pairs will be near $0$ or $1$ to avoid the possible extraction of half-graphs. Connections between instability and regularity in the context of model-theoretic complexity were investigated in \cite{mm3}, \cite{mm4}.

The arguments below give several distinct regularity lemmas, and thus a flavor of the utility of model-theoretic technology in analyzing regularity.  
For instance, in the usual proof of Szemer\'edi's lemma, the argument from mean-square density allows for the construction of a partition in which the interaction of different pieces is regular; but this need not be because the pieces themselves are necessarily atomic or uniform, in their own right. By contrast, a recurring feature of the proofs in this paper is the use of stability theory to construct partitions in which the pieces themselves have a certain inherent indivisibility; one can then obtain the generic interaction of these pieces (regularity, uniformity) nearly for free.
Actually, with the exception of \S \ref{s:ind}, we do not require the full first-order theory of the graph to be stable, just that the formula $xRy$ have the non-$k$-order property for some finite $k$. 

We now describe the structure of the paper.
Section \ref{s:pre} contains basic definitions, properties, and notation. 
We then develop a series of partition theorems with different features, illustrating certain tradeoffs between
indivisibility of the components, uniformity of their interaction, number of components and irregularity. 
Section \ref{s:ind} applies an essential feature of stability, the existence of relatively large indiscernible sequences. 
We first prove, in Theorem \ref{ind-stable}, that in a finite stable graph one can extract much larger indiscernible sets than 
would be expected from Ramsey's theorem. The main result of $\S 3$, Theorem \ref{ind-theorem}, 
applies \ref{ind-stable} to obtain an equitable partition of any stable graph in which the number of pieces grows with the size of the graph; 
however, the pieces themselves are indiscernible, and all pairs interact in a strongly uniform way not superceded in later sections. 
In Section \ref{s:dependent}, we discuss some extensions of these ideas to the wider class of dependent graphs, defined there.
Section \ref{s:f-stab} takes a different approach to the partition of stable graphs, aimed towards addressing 
the tower-of-exponential bounds. Here too the size of the partition
depends on the size of the graph, as the ``indivisibility'' condition remains quite strong 
($\epsilon$-indivisible, Definition \ref{f-stb}). We prove 
two different partition results, Theorem \ref{a23} and Theorem \ref{ind-new}. 
Theorem \ref{a23} gives an equitable partition of a given graph $A$, with $|A| = n$, into $\epsilon$-indivisible pieces; 
there is a remainder of size no more than $n^\epsilon$, and a small number of 
``irregular'' pairs, however the ``regular'' pairs have \emph{no} exceptional edges and the total number of pieces is approximately $n^c$ where
$c=c(\epsilon)=1-\epsilon^{k_{**}+1} - 2\epsilon^{2k_{**}+1}$, for $k_{**}$ a stability constant from Definition \ref{tree-bound}. 
A more combinatorial approach resulting in Theorem \ref{ind-new} allows for no irregular pairs at the cost of a larger remainder. 
Section \ref{s:order} contains the result mentioned in the paper's abstract, Theorem \ref{k26}; its proof does not depend on earlier sections. Theorem \ref{k26} is a stronger version of Szemer\'edi regularity for stable graphs: the result gives an equitable partition of any sufficiently large stable graph into a small number $m$ of pieces ($m = m(\epsilon)$, and $m \leq (3+\epsilon)\left(\frac{8}{\epsilon}\right)^{k_{**}}$ for $\epsilon$ sufficiently small), such that each of the pieces satisfies an indivisibility condition (Definition \ref{good}) and all of the pairs are $\epsilon$-uniform (Claim \ref{nice}), a stronger condition than $\epsilon$-regularity applicable as the densities are all near $0$ or $1$. To conclude, Corollary \ref{k26} gives a slightly weaker statement of Theorem \ref{k26} using the terminology of regularity. 

Sections \ref{s:ind}-\ref{s:order} can be read independently. 
The authors are working on improving the bounds in \ref{ind-claim} and \ref{ind-theorem} 
and in \ref{discussion} and on the parallel to \S \ref{s:order} for $k_*$-dependent graphs (necessarily with exceptional pairs).

\section{Preliminaries}
\label{s:pre}

\begin{ntn} \emph{(Graphs)}
We consider graphs model-theoretically, that is, as structures $G$ in a language with equality and a symmetric irreflexive binary relation $R$, 
whose domain consists of a set of vertices, and where the interpretation $R_G$ consists of all pairs of vertices $(a,b)$ connected by an edge. 
We will often write $aRb$ to indicate that $(a,b) \in R_G$, and write $G$ for the domain of $G$. 
In particular, $|G|$, the cardinality of $G$, is the number of vertices. 
\end{ntn}

\begin{defn} \label{non-op} \emph{(The non-$k$-order property)}
A graph $G$ has the \emph{non-$k$-order property} when there are no $a_i, b_i \in G$ for $i<k$ such that
$i<j < k \implies (a_i R_G b_j) \land \neg(a_j R_G b_i)$.  If such a configuration does exist, $G$ 
has the \emph{$k$-order property.}
\end{defn}

\begin{rmk} By the symmetry of $R$, it is enough to rule out the order in one direction (i.e. the non-$k$-order propery also implies that for no such sequence does $i<j \implies \neg (a_i R_G b_j) \land (a_j R_G b_i)$. 
\end{rmk}

\begin{claim} \label{a4}
Suppose $G$ is a graph with the non-$k$-order property. 
Then for any finite $A \subseteq G$, $| \{ \{ a \in A : aR_G b \} : b \in G \} | \leq |A|^k$, more precisely $~\leq~ \Sigma_{i\leq k} \binom{|A|}{i}$. 
\end{claim}

\begin{proof}
See \cite{Sh:c} Theorem II.4.10(4) p. 72 and Theorem 1.7(2) p. 657. 
\end{proof}

\begin{defn} \label{ind-seq} \emph{(Indiscernibility)}
Let $M$ be a model, let $\Gamma$ be a set of formulas in the language of $M$ and $\alpha$ an ordinal. 
Recall that a sequence $\langle a_i : i<\alpha \rangle$ of elements of $M$ is said to be a \emph{$\Gamma$-indiscernible sequence}
if for any $n<\omega$, any formula $\gamma=\gamma(x_0,\dots x_{n-1}) \in \Gamma$ and any two increasing sequences
$i_0 < \dots < i_{n-1}$, $j_0 < \dots < j_{n-1}$ from $\alpha$, we have that $M \models \gamma(a_{i_0},\dots a_{i_{n-1}})$ iff $M \models \gamma(a_{j_0}, \dots a_{j_{n-1}})$. 
\end{defn}

\begin{ntn} 
Let $\vp$ be a formula. Then we identify $\vp^0 = \neg \vp$, $\vp^1 = \vp$. We also identify ``true'' with $1$ and ``false'' with $0$, so that
in particular the intended interpretation of $\vp^X$, where $X$ is an expression which evaluates to either true or false, is simply
$\vp$ or $\neg \vp$, as appropriate. 
Likewise, the intended interpretation of expressions like ``$xRa \equiv \trv$'', where $\trv \in \{0,1\}$ or is an expression which evaluates to true or false, is ``$xRa$ if and only if $\trv=1$,'' or equivalently, iff $\trv$ is true. 
\end{ntn}

\begin{defn} \label{delta-k} \emph{(The set $\Delta_k$)}
Let $\Delta_{k}$ be the set of formulas $\{ x_0 R x_1 \} \cup \{ \vp^\si_{k,m} : m\leq k, \si \in \{1,2\} \}$ where 
\[ \vp^\si_{k,m} = \vp^\si_{k,m}(x_0,\dots, x_{k-1}) = 
\left( \exists y \right) \left( \bigwedge_{\ell<m} (x_\ell R y)^{\mbox{if}~(\si=1)} \land
\bigwedge_{m\leq \ell < k} (x_\ell R y)^{\mbox{if}~(\si=2)}   \right) \]
\end{defn}

\begin{obs} \label{cut-order}
Let $H$ be a finite graph, and let $A = \langle a_i : i<\alpha \rangle$ be a $\Delta_{k}$-indiscernible sequence 
of elements of $H$ where $\alpha \geq 2k$. Suppose that for some increasing sequence of indices $i_0 < \dots < i_{2k-1} < \alpha$ and for some element $b \in H$ the following holds: 
\begin{itemize}
 \item for all $\ell$ such that $0 \leq \ell \leq k-1$, 
$b R_H a_{i_\ell}$ and 
\item for all $\ell$ such that $k \leq \ell < 2k$, $\neg b R_H a_{i_\ell}$. 
\end{itemize}
Then $H$ has the $k$-order property. 
\end{obs}

\begin{proof}
For each $m$ with $0 \leq m \leq k-1$, consider the sequence $\langle c_j : 0\leq j \leq k-1\rangle$ given by $c_j := a_{m+i}$.
Then $b$ witnesses that $\vp^1_{k,m}(c_0,\dots c_{k-1})$ is true in $H$. 
As any two increasing subsequences of $A$ of length $k$ satisfy
the same $\Delta_k$-formulas, this easily gives the $k$-order property.

Note that if we had assumed the inverse, i.e. for all $\ell$ such that $0 \leq \ell \leq k-1$, 
$\neg b R_H a_{i_\ell}$ and for all $\ell$ such that $k \leq \ell < 2k$, $b R_H a_{i_\ell}$, 
we again get the $k$-order property (the same proof works with $\vp^2$ replacing $\vp^1$, in the notation of Definition \ref{delta-k}).
\end{proof}

Below, we will consider graphs with the non-$k_*$-order property (and reserve the symbol $k_*$ for this bound).
We define an associated bound $k_{**}$ on tree height:

\begin{defn} \label{tree-bound} \emph{(The tree bound $k_{**}$)}
Suppose $G$ does not have the $k_*$-order property. 
Let $k_{**} < \omega$ be minimal so that there do not exist sequences $\overline{a} = \langle a_\eta : \eta \in {^{k_{**}}2} \rangle$ 
and $\overline{b} = \langle b_\rho : \rho \in {^{k_{**}>}2} \rangle$ of elements of $G$ such that
if $\rho^\smallfrown\langle\ell\rangle \trianglelefteq \eta \in {^{k_{**}}2}$ then $(a_\eta R b_\rho) \equiv (\ell=1)$. 
\end{defn}

\begin{rmk}
In general, given a formula $\vp(x;y)$, \emph{(i)} if $\vp$ has the non-$k_*$-order property, then $k_{**}$ exists and $k_{**} < 2^{k_*+2}-2$. Conversely \emph{(ii)} if $k_{**}$ is as in Definition \ref{tree-bound}, then $\vp$ has the non-$2^{k_{**}+1}$-order property. See Hodges \cite{hodges} Lemma 6.7.9 p. 313.
\end{rmk}

It will also be useful to speak about the average interaction of sets. 

\begin{defn} \emph{(Truth values $\trv$)}
By a \emph{truth value} $\trv=\trv(X,Y)$ for $X, Y \subset G$, we mean an element of $\{ 0,1 \}$, where these are identified with ``false'' and ``true'' respectively. When $X=\{x\}$, write $\trv=\trv(x,Y)$.  
The criteria for assigning this value will be given below. 
\end{defn}

\begin{defn} \emph{(Equitable partitions)}
We will call a partition of $A \subseteq G$ into disjoint pieces $\langle A_i : i< m \rangle$ \emph{equitable} if
for all $i<j<m$, $||A_i| - |A_j|| \leq 1$. 
\end{defn}

\begin{ntn} \emph{(Distinguished symbols)} \label{main-notation}
Throughout this article, $\epsilon, \zeta, \xi$ are real numbers in $(0,1)$. We use $\rho, \eta$ for zero-one valued sequences ${^n}{2}$, usually in the context of trees. 
(Following logical convention, a given natural number $n$ is often identified with $\{ 0, \dots n-1 \}$.)
The letters $x,y,z$ are variables, and $i,j, k,\ell, m, n$ denote natural numbers, with the occasional exception of the standard logical notation $\ell(\overline{x})$, i.e. the length of the tuple $\overline{x}$. $T$ is a first-order theory, unless otherwise specified the theory of the graph $G$ under consideration in the language (=vocabulary) with equality and a binary relation symbol $R$. 

The symbols $k_*$, $k_{**}$, $m_*$, $m_{**}$ are distinguished. When relevant (the conventions are given at the beginning of each section),
$k_*$ is such that the graph $G$ under consideration has the non-$k_*$-order property, Definition \ref{non-op}, and $k_{**}$ is the associated tree bound, Definition \ref{tree-bound}. 
(The one exception is \S \ref{s:dependent}, in which $k_*$ is such that the graph under consideration is $k_*$-dependent.) 
The relevant sections all compute bounds based on $k_*$, so it is useful, but not necessary, to assume $k_*$ is minimal for this property. Likewise, various arguments in the paper involve construction of a rapidly decreasing sequence $\langle m_\ell : \ell < k_{**} \rangle$ of natural numbers. In context, we use $m_*$ and $m_{**}$ to refer to the first and last elements of the relevant sequence, i.e., $m_0$ and $m_{k_{**}-1}$ respectively. 

$G$ is a large graph, usually finite; $A,B,X,Y...$ are finite subgraphs of the ambient $G$. Alternately, one could let $G$ be infinite, while restricting consideration to its finite subgraphs. 
\end{ntn}

\section{A partition into indiscernible pieces}
\label{s:ind}

Classically, the hypothesis of stability implies that in infinite models, one can extract large indiscernible sequences 
(Definition \ref{ind-seq} above).
More precisely, given $\lambda$ an infinite cardinal, $M$ a model whose theory is stable in $\lambda$ and $A, I \subseteq M$
with $|A| \leq \lambda < |I|$, there is $J \subseteq I$, $|J| > \lambda$ such that $J$ is
an $A$-indiscernible sequence (in fact, an $A$-indiscernible set); see \cite{Sh:c} Theorem 2.8 p. 14.

In this section, we begin by proving a finite analogue of this result, Theorem \ref{ind-stable}, 
which shows that a finite stable graph 
will have relatively large indiscernible subsequences (in fact, subsets)
compared to what one could expect from Ramsey's theorem. We apply this to give an equitable partition 
of any stable graph in which the number of pieces is much larger than the size of those pieces; 
the gain, however, is that the pieces in the partition are themselves indiscernible sets, 
there are no irregular pairs, and the condition of ``regularity'' is very strong. 
Namely, to each pair of pieces $(A,B)$ we may associate a truth value $\trv_{A,B}$ such that there are at most a constant number of exceptional edges ($aRb \not\equiv \trv_{A,B}$). This is not superceded in later sections. Moreover we can extend some results to unstable dependent theories $T$, see \S \ref{s:dependent}.

\begin{hyp} \label{main-hypothesis}
Throughout \S \ref{s:ind} $G$ is a finite graph with edge relation $R$ which has the non-$k_*$-order property.
\end{hyp}

The next claim will be applied to prove Crucial Observation \ref{c-obs} below.

\begin{claim} \label{ind-claim} 
If $m \geq 4k_*$ and $\langle a_i : i<m \rangle$ is a $\Delta_{k_*}$-indiscernible sequence in $G$,
and $b \in G$, then either $|\{ i : a_i R b \}| < 2k_*$ or $|\{ i : \neg (a_i R b) \}| < 2k_*$.
\end{claim}

\begin{proof}
Suppose for a contradiction that both $Y = \{ i : a_i R b \}$ and $X = \{ i : \neg (a_i R b) \}$
have at least $2{k_*}$ elements. Let $i_1$ be the $k_*$th element of $X$ and let $i_2$ be the $k_*$th
element of $Y$. Clearly $i_1 \neq i_2$.

\emph{Case 1: $i_1 < i_2$}. By assumption, we can find a subsequence 
$a_{j_1} < \dots < a_{j_{k_*}}  < a_{j_{k_*+1}} < \dots < a_{j_{2k_*}} \leq a_m$ such that  
$\{{j_1} < \dots < {j_{k_*}} = i_1\} \subseteq X$ and $\{ i_2 = {j_{k+1}} < \dots < {j_{2k_*}} \} \subset Y$. 
Observation \ref{cut-order} gives the $k_*$-order property, contradiction.

\emph{Case 2: $i_2 < i_1$}. Similar argument, replacing $R$ by $\neg R$ (since $R$ is symmetric, it is equivalent). 
\end{proof}

\begin{defn} \label{ind-arrow} \emph{(the notation is from \cite{KaSh946})} 
Let $\Gamma$ be a set of formulas, $n_1$ a cardinal and $n_2$ an ordinal (for our purposes these will both be finite). 
Then $n_1 \rightarrow (n_2)_{T, \Gamma, 1}$ means: for every sequence $\langle a_i : i < n_1 \rangle$ 
of elements of $G$, there is a non-constant sub-sequence $\langle a_{i_j} : j<n_2 \rangle$ which is a
$\Gamma$-indiscernible sequence, Definition \ref{ind-seq}. Replacing $1$ by $\ell$ means that the tuples $a_i$ in the sequence
have length $\ell$. Usually we suppress mention of $T = Th(G)$ and assume $\ell(a_i)=1$, and
therefore simply write $n_1 \rightarrow (n_2)_\Gamma$.
\end{defn}

\begin{claim} \label{a10} 
If $n_1 \implies (n_2)^{k_*}_{2^{|\Delta_{k_*}|}}$ in the usual arrow notation then 
\[n_1 \rightarrow (n_2)_{\Delta_{k_*}}\]
\end{claim}

\begin{proof}
Given an increasing sequence of elements of $n_1$ of length $k_*$ 
we may color it according to which subset of the formulas of $\Delta_{k_*}$
hold on the sequence, and so extract a homogeneous subsequence of order type $n_2$. 
\end{proof}


As explained in this section's introduction, the advantage of the next theorem is not in showing the existence
of indiscernible subsequences, which could be obtained by Ramsey's theorem since $\Delta$ is finite, 
but rather in showing that in our context they are much larger
than expected: a priori, in Claim \ref{a10} the minimal $n_1$ is essentially $\beth_{k_*}(n_2+2^{|\Delta|})$, compared to (2)
in the theorem below. 
It is possible that versions of this result exist (for infinitary versions see \cite{Sh:c}). 
The statement and proof make use of some model-theoretic notions not needed elsewhere in the paper, i.e. types (consistent sets of formulas in the given free variables with parameters from a specified set) and R-rank (used in the ``by definition'' clause in 
Step 2A of the proof; see \cite{Sh:c} p. 21, p. 31). 

\begin{theorem} \label{ind-stable}
Assume that $k, k_2, \Delta$ are such that:
\begin{enumerate}
\item[(a)] $\Delta$ is a finite set of formulas, each with $\leq k$ free variables, and closed
under cycling the variables
\item[(b)] For each formula $\vp(x_0,\dots x_{k-1}) \in \Delta$ and any partition $\{x_0,\dots x_\ell\}, \{x_{\ell+1}, \dots x_{k-1}\}$
of the free variables of $\vp$
into object and parameter variables, the formula $\vp(x_0,\dots x_\ell; x_{\ell+1}, \dots x_{k-1})$ has the non-$k_2$-order property.
\end{enumerate}
Then:

\begin{enumerate}
\item There exists a natural number $r$ 
such that for any $A \subset G$, $|A| \geq 2$, we have that $|S_\Delta(A)| \leq |A|^{r}$

\item For each $A = \langle a_i : i<n \rangle$ there exists $u \subseteq n$ such that:
\begin{itemize}
 \item  $|u| \geq n^{(\frac{1}{2+r})^k}\cdot\left( t^{\frac{k}{(2+r)^k}} \right)^{-1}$, where $r$ is from \emph{(1)} of the theorem,
$t$ is a stability constant \emph{(}the $R$-rank of $\Delta$\emph{)} and $k$ is the number of free variables. 
 \item  $\langle a_i : i \in u \rangle$ is $\Delta$-indiscernible. 
\end{itemize}
\end{enumerate}

In particular, $n_1 \rightarrow (n_2)_{T, \Delta_{k_*}, 1}$ for any $n_1 > (c n_2)^{(2+r)^{k_*}}$, for the constant
$c=t^{\frac{k_*}{(2+r)^{k_*}}}$, depending only on $\Delta_{k_*}$, as was just described.
\end{theorem}

\begin{proof} 
(1) See \cite{Sh:c} Theorem II.4.10(4) and II.4.11(4) p. 74.

(2) Adding dummy variables if necessary, we may suppose that each $\vp \in \Delta$ has the free variables $x_0, \dots x_{k-1}$. 
(We may then have to omit $k$ elements at the end.) 

We prove by induction on $m \leq k$ that there is $u_m \subseteq n$ such that:
\begin{enumerate}
\item[(I)] $|u_{m+1}| \geq \left(\frac{|u_m|}{t}\right)^{\frac{1}{2+r}}$, where $r$ is from clause (1) of the theorem
and $t$ is a constant defined below
\item[(II)] if $i_0 < \dots < i_{k-1}$, $j_0 < \dots < j_{k-1}$ are from $u_m$, 
$\bigwedge_\ell (\ell < k-m \implies i_\ell = j_\ell)$, and $\vp \in \Delta$, then
\[ \models \vp(a_{i_0}, \dots a_{i_{k-1}}) = \vp(a_{j_0}, \dots a_{j_{k-1}})                      \]
\end{enumerate}

\noindent \emph{The case $m=0$}. Trivial: $u = n$.

\noindent \emph{The case $m+1$}. Let $u_m$ be given, and suppose $|u_m| = \ell_m$. Let $\Delta^m$ = $\{ \vp(x_0, \dots x_{m-1}, 
a_{\ell_m-m}, \dots a_{\ell_m-1}) \}$. This case will be broken up into several steps.

\step{Step 0: Arranging the elements of $u_m$ into a tree $W_*$.}
By induction on $\ell < \ell_m$ 
choose sets $W_\ell \subseteq u_m \setminus \bigcup_{j<\ell} W_j$ and a tree order $<_\ell$ on $W_{\leq \ell} = \bigcup_{j\leq \ell} W_j$
such that 
\begin{itemize}
\item $\mbox {if}~~ i<_\ell j ~~\mbox{then}~~ a_j, a_i ~~\mbox{realize the same $\Delta^m$-type over}
\{ a_\textbf{i} : \textbf{i} <_\ell i \}$
\item if $\neg (i<_\ell j)$ and $\neg(j <_\ell i)$ then $a_j, a_i$ 
{realize different $\Delta^m$-types over} $\{ a_s : s <_\ell i, s<_\ell j \}$
\end{itemize}

Call the resulting tree $W_*$, and its order $<_* := \bigcup_\ell <_{\ell}$ (so $W_*$ is $u_m$ with the tree order $<_*$). 

\step{Step 1: Choosing a branch through $W_*$ suffices, i.e. \emph{(II)} of the induction.}

Let $u_{m+1}$ be a branch through $W_*$ of maximal length (i.e. any maximal subset linearly orderd by $<_*$).
In this step, we verify that a branch indeed satisfies the inductive hypothesis on indiscernibility; Step 2 will
deal with the hypotheses on size. The key is that in every branch, the type does not depend on the last element.

More precisely, suppose $i_0 < \dots < i_{k-1}$, $j_0 < \dots < j_{k-1}$ are from $u_{m+1}$, 
$\bigwedge_\ell (\ell < k-m-1 \implies i_\ell = j_\ell)$, and $\vp \in \Delta$. 
(As we had built the tree $W_*$ by induction, $<_*$ implies $<$ in the sense of the order of the original sequence.)
Without loss of generality, suppose $i_{m-1} < j_{m-1}$.  
Then, recalling the parameters used in the definition of $\Delta^m$,
\[ \vp(a_{i_0}, \dots a_{i_{k-1}}) \iff \vp(a_{i_0}, \dots a_{i_{m-1}}, a_{\ell_m-m}, \dots a_{\ell_m-1}) \]
by inductive hypothesis, since the first $m$ indices agree.
\[\vp(a_{i_0}, \dots a_{i_{m-1}}, a_{\ell_m-m}, \dots a_{\ell_m-1}) \iff \vp(a_{i_0}, \dots a_{i_{m-2}}, a_{j_{m-1}}, a_{\ell_m-m}, \dots a_{\ell_m-1}) \]
since by construction $a_{i_{m-1}}$, $a_{j_{m-1}}$ realize the same $\Delta^m$-type over $a_{i_0}, \dots a_{i_{m-2}}$ (again, recall the parameters used), 
and finally
\[ \vp(a_{i_0}, \dots a_{j_{m-1}}, a_{\ell_m-m}, \dots a_{\ell_m-1}) \iff \vp(a_{j_0}, \dots a_{j_{k-1}}) \]
by inductive hypothesis. We have verified that
\[ \vp(a_{i_0}, \dots a_{i_{k-1}}) \iff \vp(a_{j_0}, \dots a_{j_{k-1}}) \]                    
based only on the assumption that the first $m-1$ indices coincide, which completes the inductive step. Having established that a branch
through the tree $W_*$ will give condition (II) for the inductive step, we turn to computing a lower bound on the size of a branch.

\br
\step{Step 2: Lower bounds on the length of a branch through $W_*$, i.e. \emph{(I)} of the induction.}
As we have established that any branch through $W_*$ would suffice for the inductive hypothesis (II), we now establish a lower
bound on the length of \emph{some} branch. Informally, we will call
any tree meeting the specifications of $W_*$ from Step 0 ``a $W_*$-tree''. 
Suppose we build the $W_*$-tree to be as short and wide
as possible, given the constraints of construction. 
In the calculation below, we find a number $h$ which is relatively large as a fraction of $u_m$ (i.e. Condition (I) of the induction),
and such that any maximally branching $W_*$-tree, and therefore any $W_*$-tree, 
will have a branch of size at least $h$. (It is not claimed that $h$ is optimal, but as will be seen from the construction,
it appears to be a reasonable approximation.)

This step will be split up into five parts. 

\br
\step{Step 2A: Partitioning the nodes of the tree using stability rank.}
Let $t=R(x=x,\Delta^m, 2)$ where $R$ is the stability rank; then by definition of this rank, we cannot embed ${^{t+1}2}$ in $W_*$. 
For $s\leq t$ let $S_s = \{ i \in W_{*} : \mbox{ above $i$ in the tree we can embed ${^{s}2}$ but no more} \}$. 
$W_{*}$ will be the disjoint union of $\{ S_s : s \leq t \}$, 
and if $i_1 \in S_1 \land i_2 \in S_2 \land i_i \leq_* i_2$ then $s_1 \geq s_2$.

\br
\step{Step 2B: Conditions for making $W_*$ as short as possible}. 
Let $\operatorname{ht}(i)$ be the height of node $i$ in the tree.
For $\ell \leq h$, $s \leq t$ let $S^s_\ell = \{ i : i \in S_s, \operatorname{ht}(i) = \ell \}$.
Then the shortest tree is attained when branching is maximal, i.e.
\[ S^{t-s}_\ell \neq \emptyset \iff \ell \geq s \]
By definition of $S_s$,
(a) if $i \in S_s$, $\operatorname{ht}(i)=\ell$ then in the tree above $i$, for each $\ell^\prime > \ell$ there is at most one $j \in S^{\ell^\prime}_s$. So to minimize height, we assume
(b) if $i \in S_s$, $\operatorname{ht}(i)=\ell$, then there is one immediate successor of $i$ in $S_s$, and all other immediate successors of $i$ are in $S_{s-1}$. 
For $s\leq t$, let 
\[ a_s = \operatorname{max} \{ z : \ell\leq h, i \in S^\ell_{s+1}, |X^i_{s}|=z, ~\mbox{for $X^i_s$ the set of immediate successors of $i$ in $S_s$} \} \] 
\noindent i.e. the largest number of ``new'' elements of $S_s$ which appear 
immediately following a node in $S_{s+1}$. (Note that at any given height $\ell$, 
the number of  ``maintenance'' nodes in $S_s$ whose immediate predecessor was also in $S_s$ is at most $|S^{\ell-1}_s|$ by (a).)
We will compute bounds on these constants $a_s$ in Step 2D after discussing a key inequality in 2C.

\br
\step{Step 2C: An expression for the number of nodes}. Let $h$ be the height of some maximal branching $W_*$-tree. 
Continuing the notation of Steps 2A-B, notice that if
\[  \hspace{20mm} \sum^{h}_{\ell=0} \sum^{t}_{s=1} |S^{t-s}_\ell| < |u_m|     \]
then any maximal branching $W_*$-tree of height $h$ will not exhaust the elements of $u_m$ as nodes,
since the lefthand side of the expression gives the number of nodes in the tree. In other words, 
any maximal branching $W_*$-tree, and thus \emph{any} $W_*$-tree, must have a branch of length
at least $h+1$.

\br
\step{Step 2D: Computing an inequality}. In this step, we prove that:
\[ \sum^{h}_{\ell=0} \sum^{t}_{s=1} |S^{t-s}_\ell|  \leq h^2\cdot t \cdot h^r \]
\noindent First, using (a), (b) from Step 2B:
\[ \sum^{h}_{\ell=0} \sum^{t}_{s=1} |S^{t-s}_\ell|  \leq h^2(\sum_{s\leq t} a_s) \]
We can bound the values $a_s$ from 2B by noting that Part (1) of the theorem gives a bound $r$ such that there are no more than $|A|^r$ distinct
types over any given set $A$, $|A|>1$. Since we are counting siblings in the tree of a given height $\ell+1$, 
the $A$ in question is the set of common predecessors, whose size is just the height $\ell$.
That is, for any given $s$, $a_s \leq h^r$ and so:
\[ \sum^{h}_{\ell=0} \sum^{t}_{s=1} |S^{t-s}_\ell|  \leq h^2\cdot t \cdot h^r \]
as desired.

\br
\step{Step 2E: Concluding that \emph{(II)} holds}. Combining 2D with the analysis of Step 2C, when $h$ satisfies
\[ \sum^{h}_{\ell=0} \sum^{t}_{s=1} |S^{t-s}_\ell|  \leq h^2\cdot t \cdot h^r < |u_m| \]
then any maximal branching $W_*$-tree, and thus \emph{any} $W_*$-tree, must have a branch of length
at least $h+1$. In particular, this inequality will hold when 
\[ h < \left|\frac{u_m}{t}\right|^{\frac{1}{2+r}} \]
\noindent Thus we conclude that
$W_{*}$ has a branch $u_{m+1}$ of length $\geq \left(\frac{|u_m|}{t}\right)^{\frac{1}{2+r}}$.

\br
\noindent\emph{This completes the inductive step.}

\br

Thus in $k$ steps we extract a sequence of indices $u$ for an indiscernible sequence; 
the size of $u$ will be at least 
$n^{(\frac{1}{2+r})^k}\cdot\left( t^{\frac{k}{(2+r)^k}} \right)^{-1}$, where $r$ is from (1) of the theorem,
$t$ is a stability constant (the $R$-rank of $\Delta$) and $k$ is the number of free variables. 

This completes the proof of the theorem.
\end{proof}

We now return to building a regularity lemma.
From Claim \ref{ind-claim} we know how individual elements interact with indiscernible sequences. The next 
observation shows a uniformity to the individual decisions made by elements in an indiscernible sequence.

\begin{obs} \emph{(Crucial Observation)} \label{c-obs}
Suppose that $A = \langle a_i : i<s_1 \rangle$, $B = \langle b_j : j<s_2 \rangle$ are $\Delta_{k_*}$-indiscernible sequences.
Suppose that $s_1 \geq 2k_*$ and $s_2 > (2k_*)^2$. 
\\Let $\uu = \{ i < s_1 : \exists^{\geq 2k_*} j < s_2) (a_{j} R b_{i}) \}$. 
Then either $|\uu| \leq 2k_*$ or $|\uu| \geq s_1-2k_*$. 
\end{obs}

\begin{proof}
Suppose the conclusion fails. Let $i_1$ be the $k_*$th member
of $\uu$, and let $i_2$ be the $k_*$th member of $\{0,\dots s_1-1\} \setminus \uu$. Clearly $i_1 \neq i_2$.

\noindent\emph{Case 1:} $i_1 < i_2$. 
Choose elements $j_0 < \dots < j_{k_*-1}$ from $\uu$ and 
elements $j_{k_*} < \dots < j_{k_*+ k_*-1} < s_1$ from $\{ 0, \dots s_1-1 \} \setminus \uu$ satisfying $j_{k_*-1} \leq i_1 < i_2 \leq j_{k_*}$. 
Recall by Claim \ref{ind-claim} that each $a_{j_\ell}$ partitions $B$ into a small and large set;
for each $\ell < 2k_*$, let the ``small set'' be 
\[ W_\ell = \{ i < s_2 :~~ a_{j_\ell} R b_{i} \leftrightarrow \left( (\exists^{\geq 2k_*} i < s_2) \neg (a_{j_\ell} R b_{i})\right) \} \]
By Claim \ref{ind-claim} and the definition of $\uu$, each $|W_\ell| < 2k_*$. Thus
\[   \left| \bigcup_{\ell<2k_*} W_\ell \right| \leq (2k_*)^2 < |B| \]
Choose $n \in \{ 0, \dots s_2 -1 \} \setminus \bigcup_{\ell<2k_*} W_\ell$. 
Then for all $\ell$ such that $0 \leq \ell \leq k_*-1$, 
$b_n R a_\ell$ and for all $\ell$ such that $k_* \leq \ell < 2k_*$, $\neg b_n R a_\ell$. By Observation \ref{cut-order}, 
$G$ has the $k_*$-order property, contradiction.

\noindent\emph{Case 2:} $i_2 < i_1$. Similar, interchanging $R$ and $\neg R$.   
\end{proof}

\begin{concl} \label{ind-concl}  
Recall the hypotheses of this section: $G$ is a finite graph with the non-$k_*$-order property.

If (A) then (B).

\begin{enumerate}
\item[(A)] 
\begin{itemize}
\item[(1)] $n_1 \rightarrow (n_2)_{T, \Delta_{k_*}, 1}$
\item[(2)] $n > n_1 n_2$ and $n_2 \geq (2k_*)^2$  
\end{itemize}

\br
\item[(B)] if $A \subseteq G$, $|A| = n$, then we can find $\overline{A}$, $m_1, m_2$ such that:
\begin{itemize}
\item[(a)] $\overline{A} = \langle A_i : i<m_1 \rangle$ 
\item[(b)] $\overline{A}$ is a partition of $A$
\item[(c)] $n = n_2 m_1 + m_2$, 
$m_2 < n_1 \leq m_1$ 
\item[(d)] For each $i$, $|A_i| \in \{ n_2, n_2 + 1 \}$
\item[(e)] Each $A_i$ is either a complete graph or an empty graph (after possibly omitting one element)
\item[(f)] If $i \neq j < m_1$, then (after possibly omitting one element of $A_i$ and/or $A_j$),
for some truth value $\trv(A_i A_j) \in \{0,1\}$, we have that for all but $\leq 2k_*$ $a\in A_i$,
for all but $\leq 2k_*$ $b \in A_j$, $aRb \equiv \trv(A_i, A_j)$.
\end{itemize}
\end{enumerate}
\end{concl}

\begin{proof} 
First, choose $m_1$ satisfying $n_2 m_1 \leq n < n_2 m_1 + n_1$ (so $m_1 \geq n_1$ by (A)(2)), and
let $<_*$ be a linear order on $A$. Second, we choose $A^\prime_{\ell}$ by induction on $\ell<m_1$ to satisfy:

\begin{itemize}
\item $A^\prime_\ell \subseteq A \setminus \{ A^\prime_j : j < \ell \}$
\item $|A^\prime_\ell| = n_2$
\item if we list the elements of $A$ in $<_*$-increasing order as $\langle a_{\ell,i} : i<n_2 \rangle$, this is 
a $\Delta_{k_*}$-indiscernible sequence. 
\end{itemize}

The existence of such $A^\prime_\ell$ is guaranteed by the hypothesis (A)(1). Since $\Delta_{k_*}$
includes $\{x R y\}$, we will have (e) by the symmetry of $R$. 

Third, let $\langle a^*_i : i < m_2 \rangle$ list the remaining elements, i.e. those of
$A \setminus \bigcup\{A^\prime_\ell : \ell < m_1 \}$. Let $A_\ell := A^\prime_\ell \cup \{a^*_\ell \}$ if $a^*_\ell$ is well defined
and $A_\ell := A^\prime_\ell$ otherwise. Condition (c) ensures there is enough room.  
This takes care of (a)-(e). 

In Condition (f), we may want to delete the extra vertex added in the previous paragraph. Then the 
Crucial Observation \ref{c-obs} applied to any pair $(A_i, A_j)$ gives our condition, i.e. it shows that 
if we choose an element $a \in A_i$ (provided we did not choose one of the at most $2k_*$ exceptional points) 
and then subsequently choose an element $b \in A_j$ 
(all but at most $2k_*$ of them are good choices) we find that $a,b$ will relate in the expected way. 
This completes the proof.
\end{proof}


\begin{theorem} \label{ind-theorem} 
Let $k_*$, $n_2$ be given with $n_2 > (2k_*)^2$. Then there is $N=N(n_2, k_*)$ such that any finite graph $G$, 
$|G|>N$ with the non-$k_*$-order property admits a partition $\overline{G} = \langle G_i \rangle$ 
into disjoint pieces $G_i$ which satisfies:
\begin{enumerate}
\item for each $G_i \in \overline{G}$, $|G_i| \in \{ n_2, n_2+1 \}$
\item (after possibly omitting one element) each $G_i$ is either a complete graph or an empty graph
\item for \emph{all} pairs $G_i, G_j \in \overline{G}$ (after possibly omitting one element from each) there exists a truth
value $t(G_i, G_j) \in \{0,1\}$ such that for all but $\leq 2k_*$ $a\in G_i$,
for all but $\leq 2k_*$ $b \in G_j$, $aRb \equiv \trv(G_i, G_j)$.
\end{enumerate}

\noindent Moreover, $N=n_1n_2$ suffices for any $n_1 > (c n_2)^{(2+r)^{k_*}}$, as computed in Theorem
\ref{ind-stable} in the case where $\Delta = \Delta_{k_*}$ \emph{(}in that calculation 
$c=t^{\frac{k_*}{(2+r)^{k_*}}}$ was a constant depending only on $\Delta$, i.e. $\Delta_{k_*}$\emph{)}. 
\end{theorem}

\begin{proof}
By Conclusion \ref{ind-concl} and Theorem \ref{ind-stable}.
\end{proof}

\begin{rmk}
\begin{enumerate}
\item Note that clause (f) of Conclusion \ref{ind-concl} 
is stronger than the condition of $\epsilon$-regularity in the following senses.
\begin{itemize}
\item It is clearly hereditary for $C_i \subseteq A_i, |C_i| \geq |2k_*|^2$. 
\item The density of exception is small: 
\[ \frac{|\{ (a,b) \in A_i \times A_j : (aRb) \equiv \neg\trv_{i,j} \}|}{|A_i||A_j|} \leq \frac{2k_*}{|A_i|} + \frac{2k_*}{|A_j|} \]
\item If $|A_i|, |A_j|$ are not too small, $\trv_{i,j} = \trv_{j,i}$. 
\item If we weaken the condition that $\bigcup_i A_i = A$ 
to the condition that $|A \setminus \bigcup_i A_i | \leq m_2$, we can omit the exceptional points. It may be better to have
$|A_i| \in \{ n_2,1\}$ with $|\{ i : |A_i| = 1 \} < n_1$. 
\end{itemize}
\item As for the hypotheses(A)(1)-(2) of the theorem: although Theorem \ref{ind-stable} will not apply outside the stable case, 
some extensions to the wider class of dependent theories are discussed in Section \ref{s:dependent}, e.g. Claim \ref{d31}. 
\end{enumerate}
\end{rmk}


\subsection{Generalizations}

Some natural directions for generalizing these results would be the following. 
First, as stated, the Szemer\'edi condition is not \emph{a priori} meaningful for infinite sets, while the condition (f) from 
Conclusion \ref{ind-concl}
is meaningful, and we can generalize the results of this section replacing $n, n_1, n_2, k$ by infinite $\lambda_0, \lambda_1, \lambda_2, \kappa$. Second, we can allow $G$ to be a directed graph and replace $\{xRy\}$ with a set $\Phi$ of binary relations satisfying 
the non-$k_*$-order property. Third, we can replace $2$-place by $\nn(*)$-place where $\nn(*) \leq \omega$, so $\Phi$ is a set
of formulas of the form $\vp(x_0,\dots x_{n-1})$, $n<\nn(*)$. In this case, the assumption (A)(1) of Conclusion \ref{ind-concl} becomes
\[\lambda^+_1 \rightarrow _{T,\Delta}(\lambda_2)_1\]
which can be justified by appeal to one of the following:

\begin{enumerate}
\item by Ramsey: $\lambda_1\rightarrow(\lambda_2)^{<\nn(*)}_{2^{|\Delta|}}$
\item by Erd\"os-Rado, similarly
\item using Erd\"os cardinals

or
\item use stability: \cite{Sh:c} Chapter II, or better (in one model) \cite{Sh:300a} \S 5.
\end{enumerate}

Finally, it would be natural to consider extending the results above to hypergraphs.

\subsection{Remarks on dependent theories}
\label{s:dependent}

In this brief interlude we discuss some extensions of \S \ref{s:ind} to the more general class of dependent graphs, 
Definition \ref{dep}. In subsequent sections, we return to stable graphs. Although, as discussed in the introduction, the order property is enough to cause irregularity, many of the properties considered in this paper are applicable to dependent graphs, e.g. Claim \ref{a4} and Fact \ref{thm-vc}. 

Dependent theories (theories without the independence property, see below) are a rich class extending the stable theories 
(theories without the order property), and have been the subject of recent research, see e.g. \cite{Sh715} and \cite{HPP}. 
From the point of view of graph theory and combinatorics, 
the Vapnik-Chervonenkis connection \cite{mcl} makes this a particularly interesting class. 
Below, we indicate how bounds on alternation can be used to easily deduce a weaker analogue of Theorem \ref{ind-theorem}.

\begin{hyp}
$G$ is an ordered graph (or a graph) meaning that it is given by an underlying vertex set on which there is a
linear order $<^G$, along with a symmetric binary edge relation $R$. We assume that $G$ is $k_*$-dependent.
\end{hyp}

\begin{defn} \label{dep}
Let $k_* < \omega$ be given. We say that $G$ is \emph{$k_*$-dependent} when there are no $a_\ell \in G$ (for $\ell < k_*$)
and $b_u \in G$ (for $u \subseteq k_*$) such that $a_\ell R b_u$ iff $\ell \in u$. 
\end{defn}

\begin{rmk} \label{d-ind}
Stable implies dependent, i.e. if $xRy$ does not have the order property it will not have the independence property; 
but the reverse is not true. More precisely, the formula $xRy$ has the \emph{order property} if, for every $n<\omega$,
there exist elements $a_0,\dots a_n$ such that for all $i \leq n$,
\[ \models (\exists x)\left(\bigwedge_{j\leq i} \neg( xRa_j ) \land \bigwedge_{j > i} xRa_j \right) \]
(note this definition remains agnostic about the existence of an $x$ connected to some partition out of order)
whereas the formula $xRy$ has the \emph{independence property} (=is not dependent) if, for every $n<\omega$,
there exist elements $a_0,\dots a_n$ such that for all $u \subseteq n$,
\[ \models (\exists x)\left(\bigwedge_{j \in u} \neg( xRa_j ) \land \bigwedge_{j \notin u} xRa_j \right) \]
\end{rmk}

We first discuss the results of \S \ref{s:ind}. On the infinite partition theorem for dependent $T$ (existence of indiscernibles),
see \cite{KaSh946} for negative results, and \cite{Sh950} for positive results.

\begin{ntn} \begin{enumerate}
\item If $A \subseteq G$ then let $\mem_A(\ell)$ be the $\ell$th member of $A$ under $<^G$, for $\ell < |A|$
\newline (for infinite $A$ such that $<^G|_A$ is well ordered, $\ell < (\operatorname{otp}(A), <^G|_A)$, an ordinal)
\item If $A \subseteq G$, $<^G$ is a well-ordering of $A$ then $\mem_A(\ell)$ is defined similarly.
\end{enumerate}
\end{ntn}

\begin{defn} \emph{(Compare Observation \ref{c-obs}.)}
\begin{enumerate}
\item We say a pair $(A,B)$ of subsets of $G$ is \emph{half-$\fb$-nice} (for $\fb=(f_1,f_2)$) when:
\\ for all but $< f_1(|A|)$ members of $A$,
\\ for all but $< f_2(|B|)$ numbers $\ell < |B|$ \emph{(or $\operatorname{otp}(B)$)}, we have that 
\\ $aR \mem_B(\ell) \equiv aR\mem_B(\ell+1)$.
\item (restated for clarity:) 
If $\fb=(c_1,c_2)$ where $c_1, c_2$ are constants, then in (1) replace the condition ``all but $\leq f_i(|A|)$ members of $A$''
with ``all but $\leq c_i$ members of $A$ for $i=1,2$. 
\item We say a pair $(A,B)$ of subsets of $G$ is \emph{$\fb$-nice} when $(A,B)$ and $(B,A)$ are both half-$\fb$-nice.
\item If $G$ is just a graph then the above $A,B$ should be replaced by $(A, <_A)$, $(B, <_B)$. 
\end{enumerate}
\end{defn}

\begin{defn}
Let $\Delta_{k_*} = \{ \vp_\eta(x_0, \dots x_{k_*-1}) : \eta \in {^{k_*}2} \}$ where
\[ \vp_\eta(x_0, \dots x_{k_*-1}) = (\exists y) \bigwedge_{\ell < k_*} (x_\ell R y)^{\mbox{if}~ \eta(\ell)=1} \]
\end{defn}

\begin{claim}
Suppose $A,B \subset G$ are disjoint, both $A,B$ are $\Delta_{k_*}$-indiscernible sequences, and $|A| \geq 2k_*$, $|B| \geq 2k_*$. 
Then $(A,B)$ is $(1,k_*)$-nice.
\end{claim}

\begin{proof} 
Suppose not, so without loss of generality $(A,B)$ is not half-$(1,k_*)$-nice. So there is $a\in A$ which ``alternates'' $k_*$ times
on $B$. That is, we may choose a $<^G$-increasing sequence of elements $b_{i_0}, \dots b_{i_{2k_*=1}} \subseteq B$ such that
$aRb_j$ (for $j$ even) and $\neg aRb_j$ (for $j$ odd). Let us verify that this means $G$ is $k_*$-dependent. Let
$J = \langle j_0, \dots j_{k*-1} \rangle$ be any set of indices of elements of $B$, of size $k_*$. 
Then for any $\sigma \subseteq J$, by indiscernibility, we have that
\[ \models \exists x \left( \bigwedge_{\ell \in \sigma} xRa_{j_\ell} \land \bigwedge_{\ell \in J \setminus \sigma} \neg x R a_{j_\ell} \right) \]
since this is true when the appropriate increasing sequence of $k_*$-many indices (corresponding to the pattern of membership in $\sigma$) 
is chosen from among $i_0, \dots i_{2k_*-1}$.
\end{proof}

\begin{defn} Let $\overline{A} = \langle A_i : i< i(*) \rangle$ be a set of subsets of $A \subset G$ (usually pairwise disjoint).
\begin{enumerate}
\item We say that $\overline{A}$ is \emph{half-$\fb$-nice} when $i < j < i(*)$ implies $(A_i, A_j)$ is half-$\fb$-nice.
\item We say that $\overline{A}$ is \emph{$\fb$-nice} when $i<j<i(*)$ implies $(A_i, A_j)$ is $\fb$-nice.
\end{enumerate}
\end{defn}

\begin{claim} \label{d31} Assume that
\begin{enumerate}
\item $G$ is an ordered graph and is $k_*$-dependent
\item $m_1 \rightarrow (m_2)^{\leq k_*}_{2^{|\Delta_{k_*}|}}$ in the sense of Ramsey's theorem
\item $A \subset G$, $|A| = n$ 
\end{enumerate}
Then we can find $\langle A_i : i<i(*) \rangle$ such that:
\begin{enumerate}
\item $|A_i| = m_2$ for all $i$
\item the $A_i$s are pairwise disjoint
\item each $A_i \subseteq A$ and is either complete or edge free
\item $B = A \setminus \bigcup \{ A_i : i< i(*) \}$ has $< m_1$ members
\item each $A_i$ is $\Delta_{k_*}$-indiscernible
\item $\overline{A}$ is $(1,k_*)$-nice
\end{enumerate}
\end{claim}

\begin{proof}
Straightforward.
\end{proof}

\begin{concl} Continuing with the notation of Claim \ref{d31}, let $\langle A_i : i<i(*) \rangle$ be the partition of
$A \subset G, |A| = n$ into pieces of size $m_2$ obtained there. Suppose that $m_3 | m_2$, and suppose that we divide each $A_i$
into convex intervals $A_{i,j}$ each of length $m_3$. If $m_2 > (m_3)^2$ then we obtain an equitable partition into $\frac{n}{m_3}$ pieces in which, 
moreover, for any $i_1 \neq i_2 < i(*)$,

\begin{enumerate}
\item $\{ (j_i, j_2) : (A_{i_1 j_i}, A_{i_2 j_2}) ~\mbox{is neither full nor empty} \}$ has cardinality $\leq \frac{m_2}{m_3} \times m_3 \times k_* = k_* m_2 $ (in each pair, each element's alternations can be seen in at most $k_*$ other pairs)
\item $| \{ (j_i, j_2) : j_1, j_2 < \frac{m_2}{m_3} \} | = \left( \frac{m_2}{m_3} \right)^2$
\item so the density of bad pairs is $\leq \frac{k_*m_3^2}{m_2}$
\end{enumerate}

On the other hand, the density of pairs with $i_1 = i_2$ is $\leq \frac{m_2^2}{n^2}$.  
\end{concl} 

\begin{rmk}
Here we obtain quite small pieces (coming from Ramsey's theorem) and there are exceptional pairs; but for the regular pairs there are no exceptions. 
\end{rmk}

\section{On the bounds} 
\label{s:f-stab}


In this section, we take a different approach, aimed at improving the bounds on the number of components.
First, in a series of claims, we give conditions for partitioning a given graph with the non-$k$-order property into 
disjoint $\epsilon$-indivisible sets (Definition \ref{f-stb}), and show when such sets interact uniformly. 
However, the procedure for extracting such sets does not ensure uniform size (Discussion \ref{discussion}). We solve this in two different ways. The first (probabilistic) approach, resulting in Theorem \ref{a23}, gives a partition in which there are irregular pairs, but the ``regular'' pairs have no exceptional edges. The second, resulting in 
Theorem \ref{ind-new}, proceeds by first proving a combinatorial lemma \ref{comb-lemma} which allows us to strengthen the ``indivisibility'' condition to one in which the number of exceptions is constant; thus in Theorem \ref{ind-new}, there are no irregular pairs, at the cost of a 
somewhat larger remainder. 

As mentioned in the introduction, one recurrent strategy in this paper is partitioning 
a given graph into ``indivisible'' components; compare Definition \ref{f-stb} with Definition \ref{good}. 
Reflecting the strength of Definition \ref{f-stb}, the number of pieces in each of the two partition theorems of this section
grows with the size of the graph, as in Theorem \ref{ind-theorem}. 
In Section \ref{s:order}, under the weaker Definition \ref{good}, the number of pieces in the partition will be a constant $c = c(\epsilon)$ as in the classical Szemer\'edi result. 



%
%

\begin{hyp}
Throughout \S \ref{s:f-stab}, we assume:
(1) $G$ is a finite graph. (2) $G$ has the non-$k_*$-order property, and so $k_{**}$ is the corresponding tree-height bound from Definition \ref{tree-bound}. (3) By convention $f,g$ are nondecreasing functions from $\mathbb{N}$ to $\mathbb{N}\setminus\{0\}$. 
\end{hyp}

\begin{defn} \label{f-stb} \emph{($\epsilon$- and $f$-indivisible)}
\begin{enumerate}
\item Let $\epsilon \in (0,1)_\mathbb{R}$.
We say that $A\subseteq G$ is $\epsilon$-indivisible if for every $b \in G$, for some truth value $\trv$, 
the set $\{ a \in A : aRb \equiv \trv \}$ has $< |A|^{\epsilon}$ members. 
\item In general, we say that $A$ is $f$-indivisible (where $f: \omega \rightarrow \omega$) if for any $b \in G$, 
there exists a truth value $\trv$ such that $| \{ a \in A : aRb \not\equiv \trv\} | < f(|A|)$. 
By convention in this section, we assume that $f$ is nondecreasing. 
\end{enumerate}
\end{defn}

\begin{claim} \label{c6}
Assume that $m_0 > \dots > m_{k_{**}}$ is a sequence of nonzero natural numbers and for all $\ell < k_{**}$, 
$f(m_\ell) \geq m_{\ell+1}$ (e.g. $f(n) = n^\epsilon$). If $A \subseteq G$, $|A| = m_0$ then for some $\ell < k_{**}$
there is an $f$-indivisible $B \in [A]^{m_\ell}$. 
\end{claim}

\begin{proof}
Suppose not. So we will choose, by induction on $k \leq k_{**}$, elements $\langle b_\eta: \eta \in {^{k>}2} \rangle$ 
and $\langle A_\eta : \eta \in {^{k\leq}2} \rangle$ such that:
\begin{enumerate}
\item $A_{\langle \rangle} = A$
\item $A_{\eta^\smallfrown\langle i \rangle} \subset A_\eta$
\item $A_{\eta^\smallfrown\langle 0 \rangle} \cap A_{\eta^\smallfrown\langle 1 \rangle} = \emptyset$
\item $|A_\eta| = m_{\operatorname{lg}(\eta)}$
\item $b_\eta \in G$
\item $A_{\eta^\smallfrown\langle i \rangle} = \{ a \in A_\eta : aRb_\eta \equiv (i=1) \}$
\end{enumerate}

There is no problem at $k=0$, but let us verify that the induction cannot continue past $k_{**}$. 
For each $\eta \in {^k2}$ $A_\eta \neq \emptyset$ by (4), so choose $a_\eta \in A_\eta$. If for all 
$\eta \in {^{k_{**}>}2}$ there exists $b_\eta$ such that $A_{\eta^\smallfrown\langle 1 \rangle}, A_{\eta^\smallfrown\langle 2 \rangle}$
are defined and satisfy the conditions, then the sequences 
$\langle a_\eta : \eta \in {^{k_{**}}2} \rangle$ and $\langle b_\eta : \eta \in {^{k_{**}>}2} \rangle$ contradict the
choice of $k_{**}$, Definition \ref{tree-bound}. So for at least one $\eta$, it must be that no such
$b_\eta$ can be found in $G$, i.e that for any $b\in G$, either $|\{ a \in A_\eta : aRb \}|$ or $|\{a \in A_\eta : \neg aRb \}|$
is less than $m_{\ell+1}$, for $\ell = {\operatorname{lg}(\eta)}$. Let $B = A_\eta$, so $|B|= m_\ell$ and 
$B$ is $f(|B|)$-indivisible, which completes the proof.
\end{proof}

\begin{claim} \label{c8}
Assume $f, \langle m_\ell : \ell \leq k_{**} \rangle$ are as in Claim \ref{c6}. 
For any $A \subseteq G$, we can find a sequence $\langle A_j : j < m \rangle$
such that:

\begin{enumerate}
\item[(a)] For each $j$, $A_j$ is $f$-indivisible
\item[(b)] For each $j$, $|A_j| \in \{ m_\ell : \ell \leq k_{**} \}$
\item[(c)] $A_j \subseteq A \setminus \bigcup \{ A_i : i<j \}$
\item[(d)] $A \setminus \bigcup \{ A_j : j<m \}$ has $< m_0$ members
\end{enumerate}
\end{claim}

\begin{proof}
We choose $A_j$ by induction on $j$ to satisfy (a)$+$(b)$+$(c). If $|A|<m_0$ we are trivially in case (d). 
By Claim \ref{c6}, we can continue as long as there are at least $m_0$ elements remaining. 
\end{proof}

\begin{claim} \label{c13}
Assume $\epsilon \in (0, \frac{1}{2})_\mathbb{R}$, $n^{\epsilon^{k_{**}}} > k_{**}$.
Let $\langle m_\ell : 0 \leq \ell \leq k_{**} \rangle$ be a sequence of integers satisfying
$n \geq m_0$, $m_{k_{**}} > k_{**}$ and for all $\ell$ s.t.
$0 \leq \ell \leq k_{**}$, $m_{\ell+1} = \left\lfloor(m_\ell)^\epsilon\right\rfloor$.

If $A \subseteq G$, $|A| = n$ then we can find $\overline{A}$ such that:
\begin{enumerate}
\item $\overline{A} = \langle A_i : i<i(*) \rangle$ is a sequence of pairwise disjoint sets
\item $\langle |A_i| : i<i(*) \rangle$ is $\leq$-increasing
\item for each $i<i(*)$ for some $\ell = \ell(i) < k_{**}$, $|A_\ell| = m_\ell$ and $A_\ell$ is $\epsilon$-indivisible
\item $A \setminus \{ A_i : i< i(*) \}$ has $< m_0$ elements
\end{enumerate}
\end{claim}

\begin{proof}
By Claim \ref{c8}, using $f(n) = n^\epsilon$ and renaming the sets $A_i$ so that clause (3) holds.
\end{proof}

The next claim says that for all sufficiently indivisible pairs of sets, averages exist (notice there
is a potential asymmetry in the demand that $B$ be large).

\begin{claim} \label{c10}
Suppose $A$ is $f$-indivisible, $B$ is $g$-indivisible and
and $f(|A|)\cdot g(|B|) < \frac{1}{2}|B|$. Then for some truth value $\trv = \trv(A,B)$
for all but $< f(|A|)$ of the $a \in A$ for all but $< g(|B|)$ of 
the $b \in B$, we have that $aRb \equiv \trv$.
\end{claim}

\begin{proof}
Similar to the proof of Observation \ref{c-obs} above. For each $a \in A$ there is, by $g$-indivisibility of $B$,
a truth value $\trv_a = \trv_a(a,B)$ such that $|\{ b \in B : aRb \equiv \trv_a \}| < g(|B|)$.
For $i \in \{0,1\}$, let $\uu_i = \{ a \in A : \trv_a = i \}$. If $|U_i| < f(|A|)$ for either $i$, we are done, so
assume this fails. Choose $W_i \subset \uu_i$ so that $|W_i| = f(|A|)$ for $i \in \{0,1\}$. 
Again we gather the exceptions: let
$V = \{ b \in B : (\exists a \in W_1)(\neg aRb) \lor (\exists a \in W_0)(aRb) \}$.
Then $|V| \leq (|W_1|+|W_0|)g(|B|) < |B|$ by hypothesis, so we may choose $b_* \in B \setminus V$.
But then $a \in W_1 \implies b_*Ra$ and $a \in W_0 \implies \neg b_*Ra$, contradicting the $f$-indivisibility of $A$.
\end{proof}

\begin{rmk}
When $f(n)=n^\epsilon$, $g(n)=n^\zeta$ the translated condition is: if $|A|^\epsilon|B|^\zeta < \frac{1}{2}|B|$. 
\end{rmk}

\begin{claim} \label{c15}
Let $A$ be $\zeta$-indivisible and $B$ be $\epsilon$-indivisible. 
Suppose that the hypotheses of Claim \ref{c10} are satisfied, so averages exist.
Then for all $\zeta_1 \in (0,1-\zeta), \epsilon_1 \in (0,1-\epsilon)$, we have: if $A^\prime \subset A, B^\prime \subset
B$, $|A^\prime| \geq |A|^{\zeta + \zeta_1}$, $|B^\prime| \geq |B|^{\epsilon + \epsilon_1}$, then:
\[  \left| \frac{ \{ (a,b) \in (A^\prime, B^\prime) ~:~ aRb \equiv \neg\trv(A,B) \}}{|A^\prime||B^\prime|} \right| \leq \frac{1}{|A|^{\zeta_1}} + \frac{1}{|B|^{\epsilon_1}}  \]
\end{claim}

\begin{proof}
To bound the number of exceptional edges, recall that in $A$, hence also in $A^\prime$, 
there are at most $|A|^\zeta$ elements which do not have the expected average behavior over $B$.
Likewise, for each non-exceptional $a\in A^\prime$ there are no more than $|B|^\zeta$ corresponding
exceptional points $b \in B^\prime$. Thus we compute:   
\begin{align*}
\frac{|A|^\zeta\cdot|B^\prime| + (|A^\prime| - |A|^\zeta)|B|^\epsilon}{|A^\prime||B^\prime|} & = 
\frac{|A|^\zeta}{|A^\prime|} + \left( \frac{|A^\prime|-|A|^\zeta|}{|A^\prime|} \right)\frac{|B|^\epsilon}{|B^\prime|} \\
& \leq \frac{|A|^\zeta}{|A^\prime|} + \frac{|B|^\epsilon}{|B^\prime|} \leq \frac{|A|^\zeta}{|A|^{\zeta + \zeta_1}} + \frac{|B|^\epsilon}{|B|^{\epsilon+\epsilon_1}} = \frac{1}{|A|^{\zeta_1}} + \frac{1}{|B|^{\epsilon_1}}
\end{align*}
A similar result holds for $f$-indivisible replacing $\epsilon$-indivisible. 
\end{proof}

We single out the following special case for Theorem \ref{ind-new} below.
\begin{cor} \label{c15-cor}
Let $A, B$ be $f$-indivisible where $f(n)=c$ is a constant function.  
Suppose that the hypotheses of Claim \ref{c10} are satisfied, so averages exist.
Then for all $\zeta_1 \in (0,1-\frac{c}{|A|}), \epsilon_1 \in (0,1-\frac{c}{|B|})$, we have: if $A^\prime \subset A, B^\prime \subset
B$, $|A^\prime| \geq c|A|^{\zeta_1}$, $|B^\prime| \geq c|B|^{\epsilon_1}$, then:
\[  \left| \frac{ \{ (a,b) \in (A^\prime, B^\prime) ~:~ aRb \equiv \neg\trv(A,B) \}}{|A^\prime||B^\prime|} \right| \leq \frac{1}{|A|^{\zeta_1}} + \frac{1}{|B|^{\epsilon_1}}  \]
\end{cor}

Returning to the general argument, choosing $\epsilon_1$ (here called $\zeta$) small enough means we can apply 
Claim \ref{c15} to any pair of elements from the partition in Claim \ref{c13}:

\begin{claim} \label{c17}
In Claim \ref{c13}, if $\zeta \in (0, \epsilon^{k_{**}})$, we have in addition that for every $i<j<i(*)$, 
if $A \subset A_i, |A| \geq |A_i|^{\epsilon + \zeta}$, $B \subset A_j, |B| \geq |A_j|^{\epsilon + \zeta}$ 
and $\trv_{i,j} = \trv(A_i, A_j)$ is the associated truth value, then 
\[  \left| \frac{ \{ (a,b) \in (A, B) ~:~ aRb \equiv \neg\trv_{i,j} \}}{|A||B||} \right| \leq \frac{1}{|A_i|^{\zeta}} + \frac{1}{|A_j|^{\zeta}}
\leq \frac{1}{|A|^{\zeta}} + \frac{1}{|B|^{\zeta}}  \]
\end{claim}

\begin{proof}
By Claim \ref{c15}.
Note that the enumeration along with clause (2) of Claim \ref{c13} (i.e. $|A_i| \leq |A_j|$) 
ensures $\trv_{i,j}$ is defined. 
\end{proof}

\begin{disc} \label{discussion}
In some respects Claim \ref{c17}, applied to the partition of Claim \ref{c13}, is quite strong: 
(a) There are no irregular pairs. (b) For each pair the number of exceptional edges is
very low. On the other hand: (c) There is a remainder $A \setminus \bigcup_i A_i$, not serious as we can distribute the remaining elements
among the existing $A_i$ without much loss, as was done in \S \ref{s:ind}. (d) There is an inherent asymmetry: the result assumes $i<j<i(*)$, we have not discussed
$j<i<i(*)$, but this is also not serious. (e) The cardinalities of the $A_i$ are not essentially constant: this seems more serious.

We give two different resolutions of (e) in the remainder of this section. In Theorem \ref{a23}, we obtain an equitable partition at the 
price of allowing for irregular pairs. In Theorem \ref{ind-new}, we obtain much stronger indivisibility conditions on the components and no irregular pairs, at the price of a somewhat larger remainder, Theorem \ref{ind-new}. 
\end{disc}

\subsection{Towards a proof of Theorem \ref{a23}.}

\begin{defn}
Assume that $A,B$ are $f$-indivisible (usually: $\epsilon$-indivisible), disjoint (for notational simplicity), and that
$f(A) \times f(B) < \frac{1}{2} |B|$ (so $\trv(A,B)$ is well defined). Let $m$ divide $|A|$ and $|B|$. 

We define a probability space: divide $A$ into $|A|/m$ pieces each of size $m$ $\langle A_i : i<i_A \rangle$
and likewise divide $B$ into $|B|/m$ parts each of size $m$, $\langle B_j : j<j_B \rangle$. Call this partition
an equivalence relation $E$ on $A \cup B$.  

For each $i<i_A, j<j_B$, let $\mathcal{E}^+_{A_i,A_j,m}$ be the event: for all $a \in A_i$, for all $b \in B_j$, 
$aRb \equiv \trv(A,B)$. 
\end{defn}

\begin{claim} \label{prob}
Let $A_i, A_j$ be two sets from the conclusion of Claim \ref{c13}.
So $\epsilon \in (0,\frac{1}{2})$, $f(x) = \left\lfloor x^\epsilon \right\rfloor$.
Ignoring a minor error due to rounding to natural numbers, suppose that 
$|A_i|=m_{\ell_a} = n^{\epsilon^{\ell_a+1}}$, $|A_j| = m_{\ell_b} = n^{\epsilon^{\ell_b+1}}$ and $|A_i| \leq |A_j|$. 
Let $m$ be an integer such that $m$ divides both $|A_i|$ and $|A_j|$, and $m=n^{\zeta}$ for some $\zeta < \epsilon^{k_{**}}$. 
Choose a random partition of $A_i$ and $A_j$ into pieces of size $m$. Let $A^s_i, A^t_j$ be pieces from $A_i$ and $A_j$, respectively,
under this partition. 

Then $\operatorname{Prob}(\mathcal{E}^+_{A^s_i,A^t_j,m}) \geq 1-\frac{2}{n^{\epsilon^{k_{**}}}}$. 
\end{claim}

\begin{proof}
By choice of $A_i, A_j$ we have that $\trv = \trv(A_i, A_j)$ is well defined. Let
$\uu_1 = \{ a \in A_i : | \{ b \in A_j : aRb \equiv \neg \trv \} | \geq |A_j|^\epsilon \}$, and for each $a \in A_i \setminus \uu_1$,
let $\uu_{2,a} = \{ b \in A_j : aRb \equiv \neg \trv \}$. By definition of $\trv$, $|\uu_1 | \leq |A_i|^\epsilon$ and for
each relevant $a$, $|\uu_{2,a}| \leq |A_j|^\epsilon$. 

We first consider $A^s_i$. The probability $P_1$ that $A^s_i \cap \uu_1 \neq \emptyset$ is bounded by the following:
\begin{align*}
P_1 < \frac{m|\uu_1|}{|A_i|-m}  &\leq \frac{n^\zeta |A_i|^\epsilon}{|A_i|-m} 
< \frac{n^{2\zeta} \left( n^{\epsilon^{\ell_a+2}} \right)}{n^{\epsilon^{\ell_a+1}}}\\
& \leq \frac{1}{n^{\epsilon^{\ell_a+1} - \epsilon^{\ell_a+2} -2\zeta}} = \frac{1}{n^{\epsilon^{\ell_a+1}(1-\epsilon) - 2\zeta}}
< \frac{1}{n^{\epsilon^{\ell_a+1}}} \leq \frac{1}{n^{\epsilon^{k_{**}}}} \\
\end{align*}

Now if $A^s_i \cap \uu_1 = \emptyset$ then $| \bigcup_{a \in A^s_i} \uu_{2,a} | \leq m|A_j|^\epsilon$. 
So the probability $P_2$ that we have $A^t_j \cap \bigcup_{a \in A^s_i} \uu_{2,a} \neq \emptyset$ is bounded by:
\begin{align*}
P_2 <& \frac{m|\bigcup_{a \in A^s_i}\uu_{2,a}|}{|A_j|-m} \leq \frac{m\cdot m \cdot |A_j|^\epsilon}{|A_j|-m} 
    \leq \frac{n^{2\zeta}|A_j|^\epsilon}{|A_j|-n^\zeta} \leq \frac{1}{n^{\epsilon^{k_{**}}}} \\
\end{align*}
by the analogous calculation. So 
$\operatorname{Prob}(\mathcal{E}^+_{A^s_i,A^t_j,m}) \geq (1-\frac{1}{n^{\epsilon^{k_{**}}}})^2 \geq 1-\frac{2}{n^{\epsilon^{k_{**}}}}$.
\end{proof}

\begin{claim} \label{partition}
Let $\langle m_\ell : \ell < k_{**} \rangle$ be a sequence which satisfies the hypotheses of Claim \ref{c13} 
and suppose that $m_{**}$ divides $m_\ell$ for $\ell < k_{**}$. Let $n$ be sufficiently large relative to $m_*$: it suffices
that $m_* < \frac{n}{n^{2\epsilon^{k_{**}}}}$ (see Remark \ref{m-star}).  

Let $A \subset G$, $|A| = n$ and let $\langle A_i : i<i(*) \rangle$ be the partition of $A$ given by Claim \ref{c13} with respect
to the sequence $\langle m_\ell : \ell < k_{**} \rangle$ (we will temporarily ignore the remainder of size $\leq m_*$). 
Recall that $\trv(A_i, A_j)$ is well defined for $i<j$ by Claim \ref{c10}.

Then there exists a partition $\langle C_i : i<r \rangle$ of $~\bigcup_{i<i(*)} A_i$ such that:
\begin{enumerate}
\item $\langle C_i : i<r \rangle$ refines the partition $\langle A_i : i<i(*)\rangle$
\item $|C_i| = m_{**}$ for each $i < r$ 
\item For all but at most $\frac{2}{n^{\epsilon^{k**}}} r^2$ of the pairs $(C_i, C_j)$, there are no exceptional edges: that is, if
$i<j$, $C_i \subseteq A_i$, and $C_j \subseteq A_j$, then $\{ (a,b) \in C_i \times C_j : aRb \not\equiv \trv(A_i, A_j) \} = \emptyset$. 
\end{enumerate}
\end{claim}

\begin{proof}
The potential irregularity in a pair $(C_i, C_j)$ comes from two sources. 

(a) \emph{The case where $C_i \subset A_i$, $C_j \subset A_j$, $i\neq j$ and $(C_i, C_j)$ contains some exceptional edges.} 
By Claim \ref{prob} and linearity 
of expectation, there exists a partition satisfying (1),(2), in which (3) holds when computed on pieces $C_i$, $C_j$ which came originally
from distinct components $A_i, A_j$. In fact, this will be true for all but at most 
$(1-\frac{1}{n^{\epsilon^{k_{**}}}})^2$ of such pairs by the calculation in the last line of Claim \ref{prob}.

(b) \emph{The case where $C_i, C_j$ are both from the same original component $A_i$.} Here we have no guarantee of uniformity. Let us compute a bound
on the fraction of such pairs $C_i, C_j$. The maximum is attained when all of the original components were of maximal size $m_0 = m_*$; in this case the number of ways of choosing a pair $C_i, C_j$ from the same original component is at most 
\[ \binom{\frac{m_*}{m_{**}}}{2}        \frac{n}{m_*} ~~\mbox{out of a possible}~~ \binom{\frac{n}{m_{**}}}{2}\]
so the ratio is approximately 
\[\frac{\frac{\left( \frac{m_*}{m_{**}}  \right)^2}{2} \frac{n}{m_*}} {\frac{\left( \frac{n}{m_{**}} \right)^2}{2}} = \frac{m_*}{n}\]
Recall that by hypothesis, $\frac{m_*}{n} < \frac{1}{n^{2\epsilon^{k_{**}}}}$.

Combining (a) and (b), the total fraction of irregular pairs does not exceed $\frac{2}{n^{\epsilon^{k**}}}$.
\end{proof}

\begin{rmk} \label{m-star}
In Claim \ref{partition}, the hypothesis on $m_*$ could obviously be weakened, or dropped at the expense of increasing the fraction of irregular pairs by $\frac{m_*}{n}$, as the calculation in part (b) of proof shows. 
\end{rmk}

\begin{theorem} \label{a23}
Let $\epsilon = \frac{1}{r} \in (0,\frac{1}{2})$, $k_*$ and therefore $k_{**}$ be given, and suppose $G$ is a finite graph with the non-$k_*$-order property.
Let $A \subset G$, $|A|=n$ with $n^{\epsilon^{k_{**}}} > k_{**}$. Then there is $\zeta < \epsilon^{k_{**}}$ and a partition
$\langle A_i : i<i(*) \rangle$ of $A$ such that:
\begin{enumerate}
 \item for all $i$, either $|A_i| = \left\lfloor n^\zeta \right\rfloor$ or $|A_i| = 1$
 \item $| \{ i : |A_i| = 1 \} | \leq n^\epsilon$
 \item $\frac{2}{n^{\epsilon^{k**}}} \geq \frac{1}{\binom{i(*)}{2}} \left|\{ (i,j) : |A_i|=1, |A_j|=1, ~\mbox{or}~ \{ (a,b) \in A_i \times A_j : aRb \}
\notin \{ A_i \times A_j, \emptyset \}  \} \right|$ 
\end{enumerate}
Moreover, we may choose $\zeta \geq (1-2\epsilon^{k_{**}})\epsilon^{k_{**}+1}$, so the total number of pieces $n^{1-\zeta}$ is at most
$n^c$ where $c=c(\epsilon)=1-\epsilon^{k_{**}+1} - 2\epsilon^{2k_{**}+1}$.
\end{theorem}

\begin{proof} 
Recall that $\epsilon = \frac{1}{r}$. (This hypothesis is just to ensure divisibility, and could be modified or dropped in favor of
allowing for slight rounding errors.)
Choose $m_{**}$ maximal so that $(m_{**})^{r^{k_{**}}} \leq n$, and subject to the constraint that 
$\frac{m_*}{n} < \frac{1}{n^{2\epsilon^{k_{**}}}}$. (One can drop this constraint, by Remark \ref{m-star}, 
at the cost of increasing the fraction in item (3) by $\frac{m_*}{n}$.)
By hypothesis $m_{**} > k_{**}$. Then the sequence 
$\langle m_\ell : \ell < k_{**} \rangle$ satisfies the hypotheses of Claims \ref{c6} and \ref{c13}, and furthermore $m_{**}$ divides $m_\ell$
for $\ell < k_{**}$. 
Apply Claim \ref{c13} to obtain a decomposition into $\epsilon$-indivisible pieces $A^\prime_i$ such that for each $i$ and some $\ell$, $|A^\prime_i| = m_\ell$. Claim \ref{partition} gives a further partition into pieces $\langle A_i : i<i(*) \rangle$ each of size $m_{**}$; additionally, we partition the remainder from Claim \ref{c13} into pieces of size $1$. Let $\zeta$ be such that $m_{**} = n^\zeta$. This gives clause (1), and clause (2) holds by Claim \ref{c13}(d) since $m_0 = (m_{**})^{r^{{k_{**}}-1}}$. Condition (3) holds by Claim \ref{prob}. Finally,
\[ n^\zeta = m_{**} \approx (m_{*})^{\epsilon^{k_{**}}} \approx \left( \frac{n}{n^{2\epsilon^{k_{**}}}} \right)^{\epsilon^{k_{**}}} \]
\end{proof}

\begin{rmk}
Though the number of components grows (solved only in \S \ref{s:order}) and this regularity lemma admits irregular pairs, 
the regular pairs have \emph{no} exceptional edges. 
\end{rmk}

 \subsection{Towards a proof of Theorem \ref{ind-new}.}
 
 In this subsection we take a different approach, and obtain a regularity lemma in which there are no irregular pairs, at the price
 of a somewhat larger remainder. The strategy will be to base the partition on a sequence of $c$-indivisible sets, i.e. sets
 which are $f$-indivisible for a particular constant function $f(x)=c$; such sets will then interact in a strongly uniform way.
 [Recall from Definition \ref{f-stb} that $\epsilon$-indivisible for $\epsilon \in (0,1)_\mathbb{R}$ 
 was shorthand for $f$-indivisible when $f(x)=x^\epsilon$; 
 this was the only exception to standard notation, and in particular, $c$-indivisible for $c \in \mathbb{N}$ means $f(x)=c$.] 
 The proof that such sets exist relies on a combinatorial lemma \ref{comb-lemma}. 
 To motivate the combinatorial lemma, the reader may wish to first look through the proof of the existence claim, Claim \ref{cl1}.
 
\begin{defn}
For $n,c \in \mathbb{N}$, $\epsilon, \zeta, \xi \in \mathbb{R}$ 
let $\bigoplus[n,\epsilon, \zeta, \xi, c]$ be the statement:

 For any set $A$ and a family $\emp$ of subsets of $A$, we have

 \begin{tabular}{ll}
 If \hspace{3mm} &(1) $|A| = n$ \\
 		&(2) $|\emp| \leq n^{\frac{1}{\zeta}}, \emp \subseteq \emp(A)$ \\
 		&(3) $(\forall B \in \emp)( |B| \leq n^\epsilon )$ \\
 \end{tabular}
 
 then there exists $\uu \subseteq A$, $|\uu| = \left\lfloor n^\xi \right\rfloor$ such that
 $ (\forall B \in \emp) \left(  | \uu \cap B | \leq c  \right)$.
\end{defn}

 \begin{lemma} \label{comb-lemma}
 If the reals $\epsilon, \zeta, \xi$ and the natural numbers $n,c$ satisfy: 
 
 \begin{itemize}
 \item[(a)] $\epsilon \in (0,1)$, $\zeta > 0$  
 \item[(b)] $0 < \xi < \operatorname{min}(1-\epsilon, \frac{1}{2})$
 \item[(c)] $n$ sufficiently large, i.e. $n >$ $\nn(\epsilon, \zeta, \xi, c)$ from Remark \ref{n-bound}
 \item[(d)] $c > \frac{1}{\zeta(1-\xi-\epsilon)}$  
 \end{itemize}
 
 then  $\bigoplus[n,\epsilon, \zeta, \xi, c]$ holds.  
 \end{lemma}
 
 \noindent We delay the proof until after the next claim. Note that in clause (d) we have that $c > 0$ by (b).
 
 \begin{rmk} \label{n-bound}
 In the statement of Lemma \ref{comb-lemma}, for ``$n$ sufficiently large'' it suffices to choose $n$ such that
 \[   \frac{1}{n^{1-2\xi}} + \frac{1}{n^{(1-\xi-\epsilon)c - {1}/{\zeta}}} < 1  \]
 See the last displayed equation in the proof of Lemma \ref{comb-lemma}. As explained there, the hypotheses of Lemma
 \ref{comb-lemma} imply that the two exponents are positive constants, so it is well defined to let 
 $\nn(\epsilon, \zeta, \xi, c)$ be the minimal $n \in \mathbb{N}$ for which the displayed equation is true. 
 \end{rmk}
 
 \begin{claim} \label{cl1}
 Suppose that we are given constants $k, c \in \mathbb{N}$ and $\epsilon, \xi, \zeta \in \mathbb{R}$ such that:
 \begin{enumerate}
  \item $A \subseteq G$ implies $|\{ \{ a \in A : aRb \} : b\in G \} | \leq |A|^k$
  \item $\epsilon \in (0, \frac{1}{2})$
  \item $\xi \in (0,\frac{1}{2})$ is such that $\xi < \epsilon^{k_{**}}$ and for all natural numbers $\ell \leq k_{**}$,
 $\frac{\xi}{\epsilon^\ell} < \frac{1}{2} < 1-\epsilon$. 
  \item the constant $c$ satisfies \[ c > \frac{1}{\zeta(1-\frac{\xi}{\epsilon^{k_{**}}}-\epsilon)} \]
 \end{enumerate}
 Then for every sufficiently large $n \in \mathbb{N}$ (meaning $n > \nn(\epsilon, \zeta, \xi, c)$ in the sense of Remark \ref{n-bound})
 if $A\subseteq G$ with $|A| = n$ then there is $Z \subseteq A$ such that
 \begin{enumerate}
  \item[(a)]  $|Z| = \left\lfloor n^\epsilon \right\rfloor$
  \item[(b)] $Z$ is $c$-indivisible in $G$, i.e. for any $b \in G$ there is $\trv \in \{0,1\}$ such that for 
 all but $k$ elements $a \in Z$, $aRb \equiv t$
 \end{enumerate}
 \end{claim}
 
 \begin{proof} 
Since $G$ has the non-$k_{*}$-order property (see the second paragraph of \ref{main-notation}), 
$k=k_*$ will satisfy condition (1) by Claim \ref{a4},
and recall that $k_{**}$ is the associated tree bound from \ref{tree-bound}. 
For lower bounds on the size of $n$, see Remark \ref{n-bound}.
 
 Let $A \subseteq G$, $|A|=n$ be given. For transparency of notation
 suppose that for each natural number $\ell \leq k_{**}$, $n^{\epsilon^\ell}\in \mathbb{N}$, and that $n^\xi \in \mathbb{N}$.
 We choose $m_\ell$ by induction on $\ell < k_{**}$ so that $m_{\ell+1} = \left\lfloor (m_\ell)^\epsilon \right\rfloor$. 
 By Claim \ref{c6} there is $\ell<k_{**}$ and $A_1 \subseteq A$ such that $|A_1| = m_\ell$ and $A_1$ is $\epsilon$-indivisible. 
 Let $\emp_1 = \{ \{ a \in A_1 : aRb \} : b \in G \}$. So $|\emp_1| \leq |A|^{k} = (m_\ell)^k$,
 by choice of $k$. 
 
 We would like to apply Lemma \ref{comb-lemma} to conclude that  $ \bigoplus[\epsilon,  \frac{1}{k}, \frac{\xi}{\epsilon^\ell}, c] $
 holds for $A = A_1$, $\emp = \emp_1$. Let us verify that the hypotheses of that Lemma hold: 
 
 \begin{itemize}
 \item (1),(2) hold as $|A_1| = m_\ell$, and $|\emp_1|=(m_\ell)^k = (m_\ell)^{\frac{1}{\frac{1}{k}}}$
 \item (3) holds by definition of $\emp_1$, as $A_1$ is $\epsilon$-indivisible
 \item (a) clear
 \item (b), (d) by choice of $\xi$ and $c$ in this Claim
 \item (c) by choice of $n$ ``sufficiently large''
 \end{itemize}
 
 We conclude that there is $Z \subseteq A_1$ (i.e. the $\uu$ guaranteed by Lemma \ref{comb-lemma}) 
 which is $c$-indivisible and satisfies
 \[  |Z| = \left\lfloor (m_\ell)^{\frac{\xi}{\epsilon^\ell}} \right\rfloor = 
 \left\lfloor (n^{\epsilon^\ell})^{\frac{\xi}{\epsilon^\ell}} \right\rfloor =
 \left\lfloor n^\xi \right\rfloor \]
 which completes the proof.
 \end{proof}
 
 \noindent We now prove Lemma \ref{comb-lemma}.
 
 \begin{proof} (of Lemma \ref{comb-lemma})
 Let $m = \left\lfloor n^\xi \right\rfloor$; this is the size of the set $\uu$ we hope to build.
 
 Let $\eff_* = {^mA}$ be the set of sequences of length $m$ from $A$, so $|\eff_*| = n^m$.
 We will use $\eta$ for such a sequence and write $\eta[\ell]$ for the value at the $\ell$th place.
 
 Define a probability distribution $\mu$ on $\eff \subseteq \eff_*$ by: $\mu(\eff) = \frac{|\eff|}{|\eff_*|}$.
 
 We will show that for $n > \nn(\epsilon, \zeta, \xi, c)$ in the sense of Remark \ref{n-bound}, 
 there is nonzero probability that a sequence $\eta \in \eff_*$ satisfies
 (1) all the elements of $\eta$ are distinct, i.e. as a set it has cardinality $m$ and (2) 
 for any $B \in \emp$ there are fewer than $k$ integers $t<m$ such that $\eta[t] \in B$. This will prove the lemma.
 
 We calculate the relevant probabilities in four steps. 
 
 \br
 \step{$\circledast_1$ Verifying some inequalities.} By assumption (b) of the Lemma, $1-2\xi > 0$ and
 $1-\xi-\epsilon > 0$. So by assumption (d) ${(1-\xi-\epsilon)c - \frac{1}{\zeta}} > 0$ and $c$ is a natural number. 
 We proceed to compute several probabilities.
 
 \br
 \step{$\circledast_2$ The probability that $\eta$ is not sequence of distinct elements.}
 (This is bounded by the sum over $s<t$ of the probability that $\eta[s] = \eta[t]$: note we don't mind
if this happens for more than one pair.)
 
 \[ \operatorname{Prob}\left((\exists s<t<m)(\eta[s] = \eta[t])\right) \leq \binom{m}{2}\frac{n}{n^2} \leq \frac{m^2}{2n} 
 \leq \frac{n^{2\xi}}{2n} \leq \frac{1}{2n^{1-2\xi}} < \frac{1}{n^{1-2\xi}}\]
 
 \br
 \step{$\circledast_3$ The probability that $\eta$ intersects a given $B \in \emp$ in more than $c$ places}.
 (For the bound, choose $c$ indices, then choose $c$ values for those places from $B$, over all possible choices of those values.)
 
 Let $B \in \emp$ be given. Then
 \[ \operatorname{Prob}\left((\exists^{\geq c} t<m)(\eta[t] \in B)\right) \leq \binom{m}{c} \frac{|B|^c}{n^c} \leq \frac{m^c |B|^c}{n^c}
 \leq \frac{n^{\xi c} n^{\epsilon c}}{n^c} = \frac{1}{n^{(1-\xi-\epsilon)c}}  \]
 
 \br
 \step{$\circledast_4$ The probability that $\eta$ intersects \emph{some} $B \in \emp$ in more than $c$ places}.
 By $\circledast_3$,
 \begin{align*} 
  \operatorname{Prob}\left((\exists B \in \emp)(\exists^{\geq c} t<m)(\eta[t] \in B)\right) 
 &\leq |\emp|\cdot \left(\operatorname{max}\{ \operatorname{Prob}\left((\exists^{\geq c} t<m)(\eta[t] \in B)\right) ~:~ B \in \emp \}\right) \\
 &\leq n^{\frac{1}{\zeta}} \cdot \frac{1}{n^{(1-\xi-\epsilon)c}}  = \frac{1}{n^{(1-\xi-\epsilon)c - \frac{1}{\zeta}}} \\
 \end{align*}
 As remarked above, it suffices for the Lemma to show that the sum of the probabilities $\circledast_2 + \circledast_4 < 1$, i.e. that
 \[   \frac{1}{n^{1-2\xi}} + \frac{1}{n^{(1-\xi-\epsilon)c - \frac{1}{\zeta}}} < 1  \]
 By $\circledast_1$, both exponents are nonzero, and moreover they are constant, so the sum will clearly eventually be smaller than $1$. 
 \end{proof} 

\begin{theorem} \label{ind-new}
Let $k_*$ and therefore $k_{**}$ be given. Let $G$ be a graph with the non-$k_{*}$-order property, and 
let $k=k_*$ as in the proof of Claim \ref{cl1}.

Then for any $c \in \mathbb{N}$ and $\epsilon, \zeta \in \mathbb{R}$ which, along with $k$, satisfy the
hypotheses of Claim \ref{cl1}, and any $\theta \in \mathbb{R}$, $0 < \theta < 1$, 

there exists $N = N(k_*, k, c, \epsilon, \zeta, \theta)$ such that 

for any $A \subseteq G$,
$|A|=n > N$, there is $i(*) \in \mathbb{N}$ and a partition $\langle A_i : i< i(*) \rangle$ of $A$
into disjoint pieces (plus a remainder) satisfying:

\begin{enumerate}
 \item $|A_i| = \left\lfloor n^{\theta\zeta} \right\rfloor$ for each $i< i(*)$
 \item each $A_i$ is $c$-indivisible, i.e. indivisible with respect to the constant function $f(x)=c$
 \item $|A \setminus \bigcup_{i<i(*)} A_i | \leq \left\lfloor n^{\frac{\theta}{\epsilon^{k_{**}-1}}} \right\rfloor$
\end{enumerate}
\end{theorem}

\begin{rmk}
Recall that the interaction of any two distinct $A_i, A_j$ given by this theorem will be highly uniform. Assuming $n^\theta > 2c$, average types exist in the sense of Claim \ref{c10}, and in particular the calculations of Corollary \ref{c15-cor} apply.
\end{rmk}

\begin{proof} (of Theorem \ref{ind-new})
Assume that $n$ is large enough so that $n^\theta > \nn(\epsilon, \zeta, \xi, c)+1$, where $\nn(...)$ is the lower bound from 
Lemma \ref{comb-lemma} and Remark \ref{n-bound}. Note that by choice of $k$, $k$ satisfies Claim \ref{a4}.

We are aiming for pieces of uniform size $n^{\theta\zeta}$. First, given $\theta$,
define by induction a decreasing sequence $\langle m_\ell : \ell \leq k_{**} \rangle$ by 
$m_{k_{**}-1} = \left\lfloor n^\theta \right\rfloor$, 
$m_{k_{**}} = \left\lfloor(m_{k_{**}-1})^\epsilon\right\rfloor$
and for each $1 < j \leq k_{**}$, 
$m_{k_{**}-j} = \left\lceil (m_{k_{**}-j+1})^{\frac{1}{\epsilon}} \right\rceil$. This sequence,
fixed for the remainder of the proof, satisfies the hypotheses of Claim \ref{c6}. 

Second, choose a sequence of disjoint $c$-indivisible sets $A_i$ by induction on $i$, as follows.
Let $R_i$ denote the remainder $A \setminus \bigcup_{j<i} A_j$ at stage $i$.
Apply Claim \ref{c6} to $B_i$, using the decreasing sequence $\langle m_\ell : \ell \leq k_{**} \rangle$ just defined,
to obtain an $\epsilon$-indivisible $B_i \subseteq R_i$.
By construction, for some $1 \leq \ell \leq k_{**}$ this set $B_i$ will have cardinality $m_{k_{**}-\ell} = \left\lfloor(n^{\theta})^{\frac{1}{\epsilon^{\ell-1}}}\right\rfloor$.
(Note that $\epsilon$-indivisibility need not be preserved under taking subsets.)

By the first line of the proof (recall $m_{k_{**}-1} = \left\lfloor n^\theta \right\rfloor$), we have $|B_i| > \nn(\epsilon, \zeta, \xi, c)$.
Apply Claim \ref{cl1} to $B_i$, using $c, k, \epsilon, \zeta, \xi$ as given, to extract
a $c$-indivisible subset $Z_i$ of size $|B_i|^\zeta$.
That is, $|Z_i| = |B_i|^\zeta = \left\lfloor n^{\theta \cdot \frac{1}{\epsilon^{\ell-1}} \cdot \zeta} \right\rfloor$ 
for some $1 \leq \ell \leq k_{**}$. 
Since the property of being $c$-indivisible \emph{is} preserved under taking subsets, choose $A_i$ to be any subset of
$Z_i$ of cardinality exactly $\left\lfloor n^{\theta\zeta} \right\rfloor$. This completes the construction at stage $i$. 

This construction can continue as long as the remainder $B_i$ has size at least 
$m_0 = \left\lceil n^{\frac{\theta}{\epsilon^{k_{**}-1}}} \right\rceil$, 
as required by Claim \ref{c6}; the remainder must have strictly smaller size, which completes the proof. 
\end{proof}

\section{Regularity for stable graphs}
\label{s:order}

Thus far, we have given several regularity lemmas for the class of stable (or, in Section \ref{s:dependent}, dependent) 
graphs which in some senses improved the classic Szemer\'edi result, particularly in the ``indivisibility'' of the components; however, in each
case the size of the partition given depended on $|A|$. 
In this section, we obtain a partition theorem for any graph $G$ with the non-$k_*$-order property
which unilaterally improves the usual result, Theorem \ref{m1}: for each $\epsilon$, there is $m=m(\epsilon)$ such that all sufficiently large $G$
with the non-$k_*$-order property admit an equitable distribution such that (1) there are \emph{no} irregular pairs,
(2) each component satisfies a strong indivisibility condition, called $\epsilon$-excellence, and (3) the bounds are much improved. 
For most of the construction, ``regularity'' of pairs means $\epsilon$-uniformity, Claim \ref{nice} below; this is useful in our context
as the density will be close to $0$ or $1$. A translation is given in Claim \ref{k23}, and Corollary \ref{k26} is a slightly weaker statement of the main result using the familiar definition of $\epsilon$-regularity.

This section relies on \S \ref{s:pre} (Preliminaries) for notation and definitions; nonetheless, definitions will be referenced 
the first time they are used. Although this section naturally extends the results and strategies of previous sections, 
it is self-contained and can be read independently.

\begin{hyp}
Throughout \S \ref{s:order}, we assume: \emph{(a)} $G$ is a finite graph, \emph{(b)} for some $k_*$ fixed throughout this section, 
$G$ has the non-$k_*$-order property, Definition \ref{non-op} and 
\emph{(c)} $k_{**}$ is the corresponding bound on the height of a 2-branching tree, Definition \ref{tree-bound}.
Throughout this section $\epsilon, \zeta, \xi$ are reals $\in (0,\frac{1}{2})$.
\end{hyp}

\begin{defn} \label{good} \emph{(Good, excellent)}
\begin{enumerate}
\item We say that $A \subseteq G$ is \emph{$\epsilon$-good} when for every $b \in G$ for some truth value $\trv = \trv(b,A) \in \{0,1\}$
we have $| \{ a \in A : (aRb) \not\equiv \trv \}| < \epsilon|A|$. As $\epsilon < \frac{1}{2}$, this is
meaningful.
\item We say that $A \subseteq G$ is \emph{$(\epsilon, \zeta)$-excellent} when
\begin{enumerate}
\item[(a)] $A$ is $\epsilon$-good and moreover
\item[(b)] if $B \subseteq G$ is $\zeta$-good then for some truth value $\trv = \trv(B,A)$,
\\$| \{ a \in A : \trv(a,B) \neq \trv(B,A) \} | < \epsilon|A|$.  
\end{enumerate}
Again, as $\epsilon < \frac{1}{2}$ the average is meaningful.
When $\epsilon = \zeta$, we will just write \emph{$\epsilon$-excellent}.
\end{enumerate} 
\end{defn}

\begin{rmk}
Any set $A \subset G$ satisfying condition \emph{(b)} for $\epsilon$-excellence must also be $\epsilon$-good, 
since any singleton set $\{ b \}$ is clearly $\epsilon$-good (in fact, $\epsilon$-excellent). Any $B$ which satisfies 
$(\forall a \in G)\bigvee_{\trv \in\{0,1\}}(\forall b \in B)(aRb\equiv \trv)$ will also be excellent. 
\end{rmk}

The next claim, which will be used repeatedly, gives a way to extract 
$\epsilon$-excellent subsets of any given $A$ by inductively building a tree whose (full) branching
must eventually stop. In the statement of the Claim, Case (II) abstracts from Case (I) 
by assigning cardinalities $m_\ell$ to the levels of the tree. 

\begin{claim} \emph{(Crucial claim)} \label{k7} Assume $\epsilon < \frac{1}{2^{k_{**}}}$.
\begin{enumerate} 
\item[(I)] For every $A \subseteq G$, $|A| \geq \frac{1}{\epsilon^{k_{**}}}$, there is $A^\prime$ such that:
\begin{enumerate}
\item[(a)] $A^\prime \subseteq A$
\item[(b)] $|A^\prime| \geq \epsilon^{k_{**}-1}|A|$
\item[(c)] $A^\prime$ is $\epsilon$-excellent
\end{enumerate}
\item[(II)] Alternately, suppose we are given a decreasing sequence of natural numbers
$\langle m_\ell : \ell < k_{**} \rangle$ such that $\epsilon m_\ell \geq m_{\ell+1}$ for $\ell < k_{**}-1$, 
and $m_{k_{**}-1} > k_{**}$.  
Then for every $A \subseteq G$, $|A| \geq \frac{1}{\epsilon^{k_{**}}}$, there is $A^\prime$ such that 
\emph{(a),(b)$^\prime$,(c)$^\prime$} hold, where:
\begin{enumerate}
 \item[(b)$^\prime$] $|A^\prime| = m_\ell$ for some $\ell < k_{**}$
 \item[(c)$^\prime$] $A^\prime$ is $\frac{m_{\ell+1}}{m_{\ell}}$-excellent \emph{(}so in particular, $\epsilon$-excellent\emph{)} 
\end{enumerate}
\end{enumerate}
\end{claim}

\begin{proof}
The strategy is as follows. Since the proof is essentially the same for Cases (I) and (II), we prove both simultaneously
by giving the proof for Case (I), and pointing out when the cases differ.
We will try to choose $(\overline{A_k}, \overline{B_k})$ by induction on $k\leq k_{**}$ such that:
\begin{enumerate}
\item $\overline{A_k} = \langle A_\eta : \eta \in {^k2} \rangle$
\item $\overline{A_k}$ is is a partition of $A$, or of a subset of $A$
\item $A_{\langle \rangle} = A$
\item If $k=m+1$, $\nu \in {^m2}$ then $A_\nu$ is the disjoint union of $A_{\nu^\smallfrown\langle 0 \rangle}, A_{\nu^\smallfrown\langle 1 \rangle}$
\item $|A_\eta| \geq \epsilon^k|A|$ for $\eta \in {^k2}$ 
\newline \emph{or} In case (II): $|A_\eta| \geq m_k$, with equality if desired
\item $\overline{B_k} = \langle B_\nu : \nu \in {^{k>}2} \rangle$ ~~(note that $B_k$ is defined at stage $k+1$)
\item Each $B_\nu \subseteq G$ is $\epsilon$-good 
\newline \emph{or} In case (II): $B_\nu$ is $\frac{m_{k+1}}{m_k}$-good
\item for all $\eta \in {^{k-1}2}$, $a \in A_{\eta^\smallfrown\langle 0 \rangle}$ implies $\trv(a,B_\eta) = 0$ and
$a \in A_{\eta^\smallfrown\langle 1 \rangle}$ implies $\trv(a,B_\eta) = 1$.
\end{enumerate}

Note that $\trv(a,B_\eta)$ is well defined in (8) as $B_\eta$ is good. 
When $k=0$, define $A_{\langle \rangle} = A$. 
Now suppose $k=m+1$. In Case (I), suppose that for all $\eta \in {^m2}$, $A_\eta$ fails to be $\epsilon$-excellent. 
By definition, for each such $\eta$,
there is some set $B_\eta \subset G$ which is $\epsilon$-good and such that 
\[ \left| \{ a \in A_\eta : \trv(a,B) \neq 1 \} \right| \geq \epsilon|A_\eta| ~~\mbox{and}~~ \left| \{ a \in A_\eta : \trv(a,B) \neq 0 \} \right| \geq \epsilon|A_\eta| \]
again noting that these two sets partition $A_\eta$ by goodness of $B_\eta$.  
So we can define $A_{\eta^\smallfrown\langle i \rangle} := \{ a \in A_\eta : \trv(a,B_\eta) = i \}$ for $i \in \{0,1\}$. 

Meanwhile, in case (II), we are interested in whether $A_\eta$ is $\frac{m_{k+1}}{m_k}$-excellent rather than $\epsilon$-excellent; 
if not, there an $\frac{m_{k+1}}{m_k}$-good set $B_\eta$ such that the displayed equation holds with ``$\geq \frac{m_{k+1}}{m_k}|A_\eta|$'' in place of
``$\geq \epsilon|A_\eta|$''. In this case, choose $A_{\eta^\smallfrown\langle i \rangle}$
to be a subset of $\{ a \in A_\eta : \trv(a,B_\eta) \neq i \}$ of cardinality $m_{k+1}$, for $i=0,1$.

This completes the inductive step, and satisfies conditions (1)-(8).

We now show that the induction cannot continue indefinitely. Suppose we have defined $A_\eta$ for $\eta \in{^{k_{**}}2}$
and $B_\nu$ for $\nu \in{^{k_{**}>}2}$ satisfying (1)-(8). For each $\eta$, since we assumed either
(I) ${\epsilon^{k_{**}}} |A| > 0$ or (II) that $|A_\eta| = m_\ell \geq m_{k_{**}-1} > k_{**}$, 
we have that $A_\eta \neq \emptyset$ so we may choose $a_\eta \in A_\eta$. 
Furthermore, for each $\nu \in {^{k_{**}>}2}$ and $\eta \in{^{k_{**}}2}$ such that $\nu \triangleleft \eta$, 
we may define
\[ \mcu_{\nu, \eta} = \{ b \in B_\nu : (a_\eta R b) \not\equiv t(a_\eta, B_\nu) \} \]
\noindent i.e. the set of elements in $B_\eta$ which do not relate to $a_\eta$ in the expected way. 
By assumption $\frac{m_{k+1}}{m_k} \leq \epsilon$, so in both Cases (I) and (II),
$|\mcu_{\nu, \eta}| < \epsilon|B_\nu|$ by the goodness of $B_\nu$. Hence for any such $\nu$,
\[  \left| \bigcup \{ \mcu_{\nu, \eta} : \nu \triangleleft \eta \in{^{k_{**}}2} \} \right| < 2^{k_{**}}\epsilon|B_\nu| < |B_\nu|  \] 
by the hypothesis on the size of $\epsilon$. In particular, for each $\nu \in {^{k_{**}>}2}$ we may choose an element
$b_\nu \in B_\nu \setminus \bigcup \{ \mcu_{\nu, \eta} : \nu \triangleleft \eta \in{^{k_{**}}2}\}$. 
Now the sequences $\langle a_\eta : \eta \in {^{k_{**}}2} \rangle$ 
and $\langle b_\nu : \nu \in {^{k_{**}>}2} \rangle$ contradict Definition \ref{tree-bound}, i.e. 
the choice of $k_{**}$.

We have shown that for some $k<k_{**}$ the induction must stop. Hence for some $\nu \in {^k2}$, $A_\nu$ is $\epsilon$-excellent
[if in case (II), $A_\nu$ is $\frac{m_{k+1}}{m_k}$-excellent, so in particular $\epsilon$-excellent] 
and satisfies condition (5), which completes the proof.
\end{proof}

\begin{rmk}
Note that the tree construction just given naturally tends away from uniform size since we do not know
when or where the induction will stop.
\end{rmk}

By definition, if $A$ is $\epsilon$-excellent and $B$ is $\zeta$-good, they will interact in a strongly uniform way, namely, 
most of the elements of $A$ will have the same average $\trv(a,B) \in \{0,1\}$ over $B$. Let us give this a name:

\begin{claim} \label{nice}
If $A$ is $\epsilon$-excellent and $B$ is $\zeta$-good then the pair $(A,B)$ is \emph{$(\epsilon, \zeta)$-uniform}, 
where we say that \emph{$(A,B)$ is $(\epsilon, \zeta)$-uniform} if for some truth value $\trv=\trv(A,B) \in \{0,1\}$
we have:
for all but $< \epsilon|A|$ of the elements of $A$, $\trv(A,B) = \trv(a,B)$.

In other words, for all but $<\epsilon|A|$ of the elements of $|A|$, for all but $< \zeta |B|$ of the elements of $B$,  $(aRb) \equiv (\trv(A,B)=1)$. When $\epsilon = \zeta$, we will just write \emph{$\epsilon$-uniform}. 
\end{claim}

\begin{proof}
By the definition of \emph{excellent}.
\end{proof}

\begin{rmk}
So in some ways ``$(A,B)$ is $(\epsilon, \epsilon)$-uniform'' is stronger than being $\epsilon$-regular; see also Claim \ref{k23} below. 
\end{rmk}

\begin{disc}
At this point, we have a way to obtain $\epsilon$-excellent subsets of any given graph, whose sizes vary along a fixed sequence. 
Below, we will extract a collection of such sets as the first stage in obtaining a regularity lemma. However, the goal is a
partition into pieces of approximately equal size, 
which will require an appropriate further division of the first-stage collection of $\epsilon$-excellent sets. 
In preparation, then, we now apply several facts from probability to prove that sufficiently large $\epsilon$-excellent 
sets can be equitably partitioned into a small number of pieces 
all of which are $\epsilon^\prime$-excellent for $\epsilon^\prime$ close to $\epsilon$.
\end{disc} 

\begin{fact} \label{normal}
Assume $p,q > 0$. If $|A| = n$, $B\subset A = p$, $m \leq n$, $\frac{m}{n} \geq q$, 
$A^\prime$ is a random subset of $A$ with exactly $m$ elements, \emph{then}
\[ \operatorname{Prob}\left( \frac{|A^\prime \cap B|}{|A^\prime|} \in \left( \frac{|B|}{|A|}-\zeta, \frac{|B|}{|A|}+\zeta \right) \right)  \]
can be modeled by a random variable which is asymptotically normally distributed. 
\end{fact}

\begin{proof}
That is, our hypergeometric distribution (sampling $m$ elements from a set of size $n$ without replacement) 
will be asymptotically approximated by the binomial distribution (sampling with replacement), and therefore by the normal distribution. 
See Erd\"os and R\'enyi \cite{erdos-renyi} p. 52, Feller \cite{feller} p. 172, Nicholson \cite{nicholson}. Note that in our case $m$ will remain relatively large as a 
fraction of $n$. 
\end{proof}

\begin{fact} \label{thm-vc} \emph{(Vapnik and Chervonenkis, \cite{vc})}
Let $X$ be a set of events on which a probability measure $P_X$ is defined. Let $S$ be a collection of random events,
i.e. subsets of $X$, measurable w.r.t. $P_X$. 
Each sample $x_1,\dots x_\ell$ and event $A \in S$ determines a relative frequency
$v^{(\ell)}_A$ of $A$ in this sample. Let $P(A)$ be the probability of $A$ and 
let $\pi^{(\ell)}= \operatorname{sup}\{ |v^{(\ell)}_A - P(A)| : A \in S \}$. 

For each $A \in S$ and finite sample $X_r = x_1,\dots x_r$ of elements of $X$, $A$ is said to \emph{induce}
the subset of $\{x_1,\dots x_r\}$ consisting of those elements $x_i$ which belong to $A$. 
The number of different subsamples of $X_r$ induced by sets of $S$ is denoted $\Delta^S(x_1,\dots x_r)$.
Define $m^S(r) = \operatorname{max} \{ \Delta^S(x_1,\dots x_r) \}$, where the maximum is taken over all samples
of size $r$.

Then a sufficient condition for the relative frequencies of events in $S$ to converge uniformly over $S$
(in probability) to their corresponding probabilities, i.e. for it to be true that for any $\epsilon$,
$\lim_{\ell \rightarrow \infty} \operatorname{Prob}(\pi^{(\ell)}> \epsilon) = 0$,
is that there exist a finite $k$ such that $m^S(\ell) \leq \ell^k + 1$ for all $\ell$.
\end{fact}

\begin{rmk}
The connection between the condition of Vapnik-Chervonenkis and the independence property, defined in Remark \ref{d-ind} above, 
was observed and developed by Laskowski \cite{mcl}.
\end{rmk}

\begin{fact} \emph{(Rate of the almost sure convergence)} \label{vc-bound}
\begin{enumerate}
\item \emph{(\cite{vc} p. 272)} Given $k$ from the last paragraph of Fact \ref{thm-vc}, if $\ell$ satisfies
\[ \ell \geq \frac{16}{\zeta^2}\left( k \operatorname{log}\frac{16k}{\zeta^2} - \operatorname{log}\frac{\eta}{4} \right) \]
then in any sample of size at least $\ell$, with probability at least $(1-\eta)$, the relative frequencies differ from
their corresponding probabilities by an amount less than $\zeta$, simultaneously over the entire class of events.  
\item Bounds on the error of the normal approximation to the hypergeometric distribution may be found in Nicholson
\cite{nicholson} p. 474 Theorem 2. 
\end{enumerate}
\end{fact}

\begin{claim} \label{k20} \emph{(Random partitions of excellent sets)} 
\begin{enumerate}
\item For every $\epsilon, \zeta$ there is $N_1$ such that for all $n > N_1 = N_1(\epsilon, \zeta)$, if $A \subset G$, $|A| = n$,
$A$ is $\epsilon$-good, $n \geq m \geq \operatorname{log}\operatorname{log}(n)$, if we randomly choose an $m$-element subset $A^\prime$ from $A$ then 
almost surely $A^\prime$ is $(\epsilon + \zeta)$-good. Moreover, we have that $b \in G \implies \trv(b,A^\prime) = \trv(b,A)$.
\item[(1A)] That is, in part \emph{(1)}, for each $\xi \in (0,1)$ there is $N_2 = N_2(\epsilon, \zeta, \xi)$ such that the probability of failure
is $\leq \xi$. 
\item Similarly for ``excellent'' replacing ``good''. 
\item In particular, for all $\epsilon^\prime >\epsilon$ and $r \geq 1$ there exists $N = N(\epsilon, \epsilon^\prime, r)$ such that if $|A|=n > N$, 
$r$ divides $n$ and $A$ is $\epsilon$-excellent, there exists
a partition of $A$ into $r$ disjoint pieces of equal size each of which is $\epsilon^\prime$-excellent.
Note that $N(\epsilon, \epsilon^\prime, r)$ increases with $r$. 
\end{enumerate} 
\end{claim}

\begin{proof} 
(1) Call $B \subset A$ an \emph{exceptional set} if there is $b \in G$ such that 
$B = \{ a \in A : aRb \not \equiv \trv(b,A) \}$ and $|B| \geq \epsilon m$. 
It suffices to show that almost surely $A^\prime$ satisfies: for all exceptional sets $B$
\[ \frac{|A^\prime \cap B|}{|A^\prime|} \in \left( \frac{|B|}{|A|}-\zeta, \frac{|B|}{|A|}+\zeta \right) \]
By Fact \ref{normal}, for $n, m$ sufficiently large, we may approximate drawing a set of size $m$ by
the sum of $m$ independent, identically and normally distributed random variables, where the probability of $x \in B$
is just $|B|/|A|$. 
Since $G$ has the non-$k_*$-order property, Claim \ref{a4} in the case where $G=A$, $A=A^\prime$ 
shows that the Vapnik-Chervonenkis sufficient conditions (Fact \ref{thm-vc}) are satisfied.
(Recall the definition of exceptional set from the first line of the proof.)

(1A) By Fact \ref{normal} and Fact \ref{vc-bound}. 

(2) Follows by the ``moreover'' in the previous clause. 

(3) Let $\epsilon$ be as given, $\zeta = \epsilon^\prime-\epsilon$, and $\xi=\frac{1}{r+1}$. 
Let us verify that $N = N_2(\epsilon, \zeta, \xi)$ suffices.
First, randomly choose a function $h: A \rightarrow \{ 0, \dots r-1 \}$ such that for all $s<r$, 
$| \{ a \in A : h(a) = s \} | = \frac{n}{r}$. Then each $s<r$ induces a random choice of a subset of $A$,
since for each $s<r$ we have $h^{-1}(s) \in [A]^{\frac{n}{r}}$. Since $h$ was random, for each given $s$,
each $B \in [A]^{\frac{n}{r}}$ is equally probable. By part (1), for each $s<t$ 
\[  1-\xi \leq \operatorname{Prob}\{ h^{-1}(s) ~\mbox{is $(\epsilon+\zeta)$-excellent} \} \]
and therefore
\[  1-r\xi \leq \operatorname{Prob}\{ \bigwedge_{s<r} h^{-1}(s) ~\mbox{is $(\epsilon+\zeta)$-excellent} \} \]
But since $1-r\xi = 1-\frac{r}{r+1} > 0$, there exists an $h$ which works, i.e. an $h$ such that for each $s<t$, 
$h^{-1}(s)$ is $(\epsilon+\zeta)$-excellent. Since $\epsilon + \zeta = \epsilon^\prime$, this finishes the proof.
\end{proof}

The next claim forms the core of the proof of Theorem \ref{m1}. 
The statement is laid out so as to make the strategy of construction clear (based on the claims established so far).
A less transparent, but more compact, list of the requirements in this claim is summarized in Corollary \ref{cor-k17}. 
For the Theorem, it remains to construct an appropriate sequence
$\langle m_i : i<k_{**} \rangle$ which respects the various bounds collected here, and to show that this can be done
while keeping $m_{**}$ sufficiently large relative to $|A|$. 

\begin{claim} \label{k17}
Assume that $\epsilon < \epsilon^\prime <  2^{-k_{**}}$. 
Suppose that $A \subseteq G$, $|A| = n$. 

\begin{enumerate}
\item Let $\langle m_i : i<k_{**} \rangle$ be a sequence of natural numbers such that $m_{i+1} \leq \epsilon m_i $
for $i < k_{**}$, and let $m_* := m_0$, $m_{**} := m_{k_{**}-1} \geq k_{**}$.
Then there is $\overline{A}$ such that:
\begin{enumerate}
\item $\overline{A} = \langle A_i : i < j(*) \rangle$, for some $j(*) \leq \frac{n}{m_{**}}$
\item For each $i$, $A_i \subseteq A$ and $|A_i| \in \{ m_\ell : \ell<k_{**} \}$
\item $i \neq j$ $\implies$ $A_i \cap A_j = \emptyset$
\item each $A_i$ is $\epsilon$-excellent
\item \underline{hence} if $i \neq j < j(*)$ then the pair $(A_i, A_j)$ is $(\epsilon, \epsilon)$-uniform
\item $B := A \setminus \bigcup \{ A_i : i<i(*) \}$ has $< m_*$ members
\end{enumerate}

\br
\item[(1A)] 
Suppose further that: 
\begin{enumerate}
\item[(i)] $m_{**} | m_k$ for each $k < k_{**}$
\item[(ii)] $m_{k_{**}-2} > N = N(\epsilon, \epsilon^\prime, \frac{m_*}{m_{**}})$  \hspace{5mm}\emph{(as in Claim \ref{k20})}  
\item[(iii)] $\operatorname{log}\operatorname{log} m_* \leq m_{**}$
\end{enumerate}
Then for some $i(*)$ with $j(*) \leq i(*) \leq \frac{n}{m_{**}}$ there is a further refinement of the partition
from \emph{(1)} into $i(*)$ disjoint pieces \emph{(}in slight abuse of notation we will now use $\langle A_i : i<i(*) \rangle$  
to refer to this new partition\emph{)} such that
for each $i<i(*)$, $|A_i| = m_{**}$. Furthermore, each of these new pieces $A_i$ is $\epsilon^\prime$-excellent.

\br
\item 
Let $\langle A_i : i<i(*) \rangle$ be the partition into equally sized $\epsilon^\prime$-excellent pieces from \emph{(1A)}. 
Then there exists a partition $\langle B_i : i<i(*) \rangle$ of the remainder $B$, allowing $B_i = \emptyset$ for some $i$ 
\emph{(}i.e. ${\left\lfloor \frac{|B|}{i(*)} \right\rfloor}$ may be $0$\emph{)} such that
\[ |B_i| \in \left\{ {\left\lfloor \frac{|B|}{i(*)} \right\rfloor}, {\left\lfloor \frac{|B|}{i(*)} \right\rfloor + 1} \right\} \]
Let $A^\prime_i = A_i \cup B_i$ for $i<i(*)$. Then:
\begin{enumerate}
\item $\langle A^\prime_i : i<i(*) \rangle$ is a partition of $A$
\item the sizes of the $A^\prime_i$ are almost equal, i.e. $||A^\prime_i| - |A^\prime_j|| \leq 1$
\item if we let 
\[ \zeta = \operatorname{max} \left\{ \frac{\epsilon^\prime|A_i| + |B_i|}{|A_i| + |B_i|} : i<i(*) \} \right\}
\leq \frac{ \epsilon^\prime m_{**} + \left\lceil \frac{m_*}{i(*)}  \right\rceil}{m_{**} + \left\lceil \frac{m_*}{i(*)}\right\rceil}  \]
\noindent then $i \neq j < i(*)$ implies $(A^\prime_i, A^\prime_j)$ is $(\zeta, \zeta)$-uniform.
\end{enumerate}

\br
\item If, moreover, $m_{**} > \frac{1}{\epsilon^\prime}$ 
and $m_* \leq \frac{\epsilon^\prime n +1}{1+\epsilon^\prime}$, then $\zeta < 3\epsilon^\prime$, 
where $\zeta$ is as in \emph{(2)(c)}. 
\end{enumerate}
\end{claim}

\begin{proof}
(1) Applying Claim \ref{k7} we try to choose a sequence of $\epsilon$-excellent sets $A_i$, each of size $m_\ell$ for some $\ell < k_{**}$,
by induction on $i$ from $C_i := A \setminus \bigcup_{j<i} A_j$. 
We can continue as long as $|C_i| \geq m_*$. Note that condition (e) is immediate, 
for all pairs $(A_i, A_j)$ without exceptions, by Claim \ref{nice}.

(1A) By Claim \ref{k20}(3). Note that in the application below, we will build all relevant sequences of $m$s to satisfy
$m_{**} \approx \epsilon^{k_{**}}m_*$ so that $N=N(\epsilon, \epsilon^\prime, \epsilon^{-k_{**}})$ can be computed, if desired,
before the sequence is chosen. 

(2) Immediate: the partition remains equitable because the $A_i$ all have size $m_{**}$, and $\zeta$ bounds the relative size
of a ``bad'' subset of any given $A_i$.

(3) Given the assumption of an equitable partition from (2)(b), it would suffice to show that
for every $i$, $|B_i| \leq 2\epsilon^\prime |A_i|$, as then we would have
\[  \frac{\epsilon^\prime|A_i| + |B_i|}{|A_i| + |B_i|} \leq \frac{\epsilon^\prime|A_i| + 2\epsilon^\prime|A_i|}{|A_i|} = 3\epsilon^\prime    \]
We verify that the assumption on $m_*$ is enough to give this bound. By definition, as the $B_i$s arise from an equitable partition
of the remainder $B$, $|B_i| \leq \frac{m_*-1}{i(*)}+1$, where $i(*)$ is the number of components from the partition (1A), by (2) above. 
Since the components $A_i$ from (1A) all have size $m_{**}$, and $|B| \leq m_* -1$, we can bound $i(*)$ by 
$\frac{n}{m_{**}} \geq i(*)  \geq \frac{n-m_*+1}{m_{**}} > \frac{n-m_*}{m_{**}}$. Thus
\[ |B_i|-1 \leq \frac{m_*-1}{i(*)} < (m_* -1)\left(\frac{n-m_*}{m_{**}}\right)^{-1} ~~\mbox{and so \hspace{3mm} } 
~~ \frac{|B_i|-1}{|A_i|} < \left(\frac{m_* -1}{m_{**}}\right)\left(\frac{n-m_*}{m_{**}}\right)^{-1} = \frac{m_*-1}{n-m_*} \]
We had assumed that $m_* \leq \frac{\epsilon^\prime n +1}{1+\epsilon^\prime}$, and so:
\begin{align*} 
m_*(1+\epsilon^\prime) &\leq \epsilon^\prime n + 1 \\
m_* - 1 & \leq (n-m_*)\epsilon^\prime \\
\frac{m_*-1}{n-m_*} &\leq \epsilon^\prime \\
\end{align*}
We had also assumed that $\frac{1}{\epsilon^\prime} < m_{**}$, i.e. $\frac{1}{m_{**}} < \epsilon^\prime$. 
Since $|A_i| = m_{**}$ (so $\frac{|B_i|-1}{|A_i|} = \frac{|B_i|}{|A_i|} - \frac{1}{m_{**}}$), we conclude that
\[ \frac{|B_i|}{|A_i|} < \frac{m_*-1}{n-m_*} + \frac{1}{m_{**}} <  \epsilon^\prime + \epsilon^\prime = 2\epsilon^\prime \]
which completes the proof. 
\end{proof}

\begin{cor} \label{cor-k17}
To summarize the requirements of Claim \ref{k17}, suppose that $k_*$ and therefore $k_{**}$ are fixed in advance, 
$G$ is a graph with the non-$k_*$-order property, and that we are given:
\begin{enumerate}
 \item $\epsilon_1, \epsilon_3 \in \mathbb{R}$ such that $0< \epsilon_3 < \epsilon_2 :=\frac{\epsilon_1}{3} < \epsilon_1 <  2^{-k_{**}}$
 \item A sequence of positive integers $\langle m_\ell : \ell < k_{**} \rangle$ such that:
 \begin{enumerate} 
    \item $m_{\ell+1} < \epsilon_3 m_\ell$ for each $\ell<k_{**}$
    \item $m_{**} | m_\ell$ for each $\ell < k_{**}$
    \item $\operatorname{log}\operatorname{log} m_0 \leq m_{**}$
    \item $m_{**} := m_{k_{**}-1} \geq \operatorname{max}(k_{**}, \frac{1}{\epsilon_2})$
    \item $m_{k_{**}-2} > N(\epsilon_3, \epsilon_2, \frac{m_0}{m_{**}})$, from Claim \ref{k20}(3)
 \end{enumerate}
  \item $A \subseteq G$, $|A| = n$ where $n$ satisfies $m_0 \leq \frac{\epsilon_2 n + 1}{1+\epsilon_2}$ 
\end{enumerate}

\noindent Then there exists \emph{$i(*) \leq \frac{n}{m_{**}}$} and a partition of $A$ into disjoint pieces $\langle A_i : i<i(*) \rangle$ such that:
\begin{itemize}
 \item  for all $i<j<i(*)$, $||A_i|-|A_j|| \leq 1$
 \item  each $A_i$ is $\epsilon_1$-excellent
 \item Each pair $(A_i, A_j)$ is $\epsilon_1$-uniform
\end{itemize}
\end{cor}

\begin{proof}
By Claim \ref{k17}, using $\epsilon = \epsilon_3$, $\epsilon^\prime=\epsilon_2$ and $3\epsilon^\prime = \epsilon_1$;
note that the partition we obtain was called $\langle A^\prime_i : i<i(*)\rangle$ in Claim \ref{k17}. 
\end{proof}

\begin{disc} 
In practice, we are given $\epsilon = \epsilon_1$, and then choose $\epsilon_3$ to run the proof of Corollary \ref{cor-k17}. 
The role of the respective $\epsilon$s appears in conditions (2)(a) and (2)(e) of this Corollary. On one hand, $\epsilon_3$ 
determines the rate of decrease of the sequence of $m$s, thus the size of $m_{**}$, and ultimately the number of components in the partition:
so one would usually want to choose $\epsilon_3$ close to $\epsilon_1 = \epsilon$. 
On the other hand as $\epsilon_3$ approaches $\epsilon_1$, the lower bound on the size of the graph $A$ may rise, via the $N$ from (2)(e),
which comes from Claim \ref{k20}(3).
\end{disc}

Before stating the main result of this section, Theorem \ref{m1},
we consider more explicitly the relation of $\epsilon$-uniformity to $\epsilon$-regularity. 
As the following calculation shows, $\eta$-uniform pairs will be
$\rho$-regular when $\rho$ (the parameter for a lower bound on the size of a subset chosen) is sufficiently large relative to $\eta$ (the 
parameter for an upper bound on the number of non-uniform edges). 
As mentioned above, uniformity is somewhat more precise in our context for large enough graphs, as the densities of sufficiently large $\epsilon$-regular pairs will be near $0$ or $1$.

\begin{claim} \label{k23}
Suppose that $\epsilon, \zeta, \xi \in (0, \frac{1}{2})$, and the pair $(A,B)$ is $(\epsilon, \zeta)$-uniform.
By uniformity, there is a truth value $\trv(A,B) \in \{0,1\}$. 
Let $Z := \{ (a,b) \in (A \times B) : aRb \not\equiv \trv \}$ and likewise let     
$Z^\prime := \{ (a,b) \in (A^\prime \times B^\prime) : aRb \not\equiv \trv \}  $.
Suppose also that $A^\prime \subseteq A$, $|A^\prime| \geq \xi |A|$, $B^\prime \subseteq B$, $|B^\prime| \geq \xi|B|$,
and $\frac{\epsilon + \zeta}{\xi} < \frac{1}{2}$. Then:
\begin{enumerate}
\item $\frac{|Z|}{|A||B|} < \epsilon + \zeta$
\item $\frac{|Z^\prime|}{|A^\prime||B^\prime|} < \frac{\epsilon+\zeta}{\xi}$ 
\end{enumerate}

\noindent In particular, if the pair $(A,B)$ is $\epsilon_0$-uniform for $\epsilon_0 \leq \frac{\epsilon^2}{2}$
then $(A,B)$ is also $\epsilon$-regular.
\end{claim}

\begin{proof}
Let $A^\prime, B^\prime$ be given. 
For $a \in A$, let $\uw_a = \{ b \in B : aRb \not\equiv \trv(A,B) \}$, and 
let $\uu = \{ a \in A : \left| \uw_a \right| > \epsilon |A| \}$.
So $|\uu| < \epsilon |A|$, and $a \in A \setminus \uu \implies |\uw_a| < \zeta|B|$.
Since
\begin{align*}
 Z \subseteq & \uu \times B \cup \bigcup\{ (a,b) \in A \times B : b \in W_a, a \notin \uu  \} \\
  Z^\prime \subseteq & \uu \times B^\prime \cup \bigcup\{ (a,b) \in A^\prime \times B : b \in W_a, a \notin \uu  \} \\
\end{align*}
we can bound the cardinalities as follows:
\begin{align*}
|Z| &\leq |\uu| \cdot |B| + |A| \cdot \operatorname{max} \{ |\uw_a| : a \in \uu \} \\
\frac{|Z|}{|A \times B|} &< \frac{\epsilon|A|}{|A|} + \frac{\zeta|B|}{|B|} = \epsilon + \zeta \\
\mbox{and likewise}& \\
\frac{|Z|^\prime}{|A^\prime \times B^\prime|} & = \frac{|\uu||B^\prime| + |A^\prime|\cdot\operatorname{max} \{ |\uw_a| : a \in \uu \}}{|A^\prime||B^\prime|}\\
& < \frac{\epsilon |A| \xi |B| + \xi |A| \zeta |B|}{|A^\prime||B^\prime|} \cdot \frac{|A||B|}{|A||B|} 
={(\epsilon \xi + \xi \zeta)} \cdot \frac{|A||B|}{|A^\prime||B^\prime|} = \frac{\xi(\epsilon + \zeta)}{\xi^2} = \frac{\epsilon + \zeta}{\xi} \\
\end{align*}
\noindent by the assumption on the size of $A^\prime,B^\prime$. This completes the proof of (1) and (2). 

For the ``in particular'' clause, let $d(X,Y) = \frac{e(X,Y)}{|X||Y|}$ be the usual edge density.
We have shown that if $\trv(A,B)=1$, 
$d(A, B) > 1-(\epsilon+\zeta)$ while $d(A^\prime, B^\prime) > 1-\frac{\epsilon + \zeta}{\xi}$, and likewise if $d(A,B)=0$,
$d(A, B) < (\epsilon+\zeta)$ while $d(A^\prime, B^\prime) < \frac{\epsilon + \zeta}{\xi}$. Thus the difference in 
density $|d(A,B) - d(A^\prime, B^\prime)|$ is bounded by $\frac{\epsilon + \zeta}{\xi}$. If $(A,B)$ is $(\epsilon_0, \epsilon_0)$-uniform 
and $\epsilon$ is such that $|A^\prime| \geq \epsilon|A|$, $|B^\prime| \geq \epsilon|B|$ where $\epsilon_0 \leq \frac{\epsilon^2}{2}$ 
then the difference in densities is bounded by $\frac{\epsilon^2}{\epsilon} = \epsilon$, which completes the proof.   
\end{proof}

We now give the main result of this section. Recall the definitions of non-$k_*$-order property (Definition \ref{non-op}), 
$k_{**}$ (Definition \ref{tree-bound}),
$\epsilon$-excellent (Definition \ref{good}), and $\epsilon$-uniform (Claim \ref{nice}). 

\begin{theorem} \label{m1} 
Let $k_*$ and therefore $k_{**}$ be given. Let $G$ be a finite graph with the non-$k_*$-order property. 
Then for any $\epsilon > 0$ there exists $m=m(\epsilon)$ such that for all sufficiently large $A \subseteq G$, 
there is a partition $\langle A_i : i<i(*)\leq m\rangle$ of $A$ into at most $m$ pieces, where:
\begin{enumerate}
 \item for all $i,j<i(*)$, $||A_i|-|A_j||\leq 1$
 \item each of the pieces $A_i$ is $\epsilon$-excellent
 \item all of the pairs $(A_i, A_j)$ are $(\epsilon, \epsilon)$-uniform
 \item if $\epsilon < \frac{1}{2^{k_{**}}}$, then $m \leq (3+\epsilon)\left(\frac{8}{\epsilon}\right)^{k_{**}}$
\end{enumerate}
\end{theorem}

\begin{proof}
Without loss of generality, assume $\epsilon < \frac{1}{2^{k_{**}}}$. (This is necessary for Claim \ref{cor-k17}, which 
uses Claim \ref{k7}.)  

We proceed in stages. Let $n=|A|$. When hypotheses are made about the minimum size of $n$, these will be labeled (Hx) and collected
in Step 5.

\step{Step 0: Fixing epsilons}. When applying Corollary \ref{cor-k17} we will use: 
$\epsilon_3 = \frac{\epsilon}{4}$, $\epsilon_2=\frac{\epsilon}{3}$, and $\epsilon_1=\epsilon$. 

\step{Step 1: Fixing $q$}. Given $\epsilon_3$, let $q = \left\lceil \frac{1}{\epsilon_3} \right\rceil \in \mathbb{N}$.
It follows that $\frac{2}{\epsilon_3} \geq q \geq \frac{1}{\epsilon_3}$ and thus $\frac{\epsilon_3}{2} \leq \frac{1}{q} \leq \epsilon_3$.
In particular, any sequence $\langle m_\ell : \ell < k_{**} \rangle$ such that $m_{**}:=m_{k_{**}-1} \in \mathbb{N}$ and
$m_{\ell} = qm_{\ell+1}$ for all $\ell < k_{**}$ will satisfy
$m_{\ell+1} = \frac{1}{q}m_\ell \leq \epsilon_3 m_\ell$, $m_\ell \in \mathbb{N}$ for each $\ell < k_{**}$, 
and $m_{**} | m_\ell$ for all $\ell < k_{**}$.

\step{Step 2: Choosing $m_{**}$}. In this step, the aim is to build a sequence $\langle m_\ell : \ell < k_{**} \rangle$ whose elements are as large
as possible subject to the constraints (2)(a),(b),(d) and (3) of Corollary \ref{cor-k17}. In keeping with prior notation, let $m_* := m_0$.
Recalling $\epsilon_2 = \frac{\epsilon}{3}$ from Step 0, Condition \ref{cor-k17}(3) asks that
\[ m_* \leq \frac{\frac{\epsilon}{3}n+1}{1+\frac{\epsilon}{3}} ~\mbox{so it suffices to choose}~ 
m_* \leq \frac{\frac{\epsilon}{3} n}{1+\frac{\epsilon}{3}} = \frac{\epsilon n}{3+\epsilon} \] 
Let (H1) be the assumption that $n$ is not too small (see Step 5). Then there exists $c \in \mathbb{N}$, $c>k_{**}$ such that
\[ q^{k_{**}-1}c \in  \left( \frac{\epsilon n}{3+\epsilon} - q^{k_{**}-1}, \frac{\epsilon n}{3+\epsilon} \right] \]
Thus setting $m_{**} := \operatorname{max} \{ c \in \mathbb{N}:  c > k_{**}, ~c > \frac{1}{\epsilon_2},~  q^{k_{**}-1}c \leq \frac{\epsilon n}{3+\epsilon} \}$ 
is well defined, and $m_{**}$ will belong to the half-open interval just given. 
Having defined $m_{**}$, for each $\ell < k_{**}$ let $m_\ell := q^{k_{**}-\ell-1} m_{**}$. By Step 1, the $m_\ell$ are integer valued and
satisfy the required conditions on divisibility and size. By choice of $c$, $m_* = q^{k_{**}-1}m_{**}$ satisfies the inequality \ref{cor-k17}(3).  

We have defined a sequence $\langle m_\ell : \ell < k_{**} \rangle$ of positive integers which satisfies conditions (2)(a),(b),(d) and (3) of
Corollary \ref{cor-k17}. We fix this sequence for the remainder of the proof, and proceed to calculate various bounds in terms of it.

\step{Step 3: Bounding ${m_{**}}$}. 
By the definition of $m_{**}$ in Step 2, $\frac{\epsilon n}{3+\epsilon} - q^{k_{**}-1} < q^{k_{**}-1}m_{**}$, so assuming $n$
is not too small [again (H1) in Step 5], 
\[ \frac{\epsilon n}{3+\epsilon}(q^{k_{**}-1})^{-1} - 1  < m_{**} ~~\implies~~ 
 \frac{1}{2}\cdot\frac{\epsilon n}{3+\epsilon}(q^{k_{**}-1})^{-1} \leq m_{**}  \]

\step{Step 4: Bounding $\frac{n}{m_{**}}$}. 
Applying Step 3, an inequality from Step 1, and the definition of $\epsilon_3$,
\begin{align*} 
\frac{n}{m_{**}} \leq & ~~ \frac{n}{\frac{1}{2}\left(\frac{\epsilon n}{3+\epsilon}\right)\left( \frac{1}{q^{k_{**}-1}}\right)} = \frac{2(3+\epsilon)q^{k_{**}-1}}{\epsilon} 
\leq  \frac{2(3+\epsilon)}{\epsilon}\left( \frac{2}{\epsilon_3} \right)^{k_{**}-1}  
= (3+\epsilon)\left(\frac{2}{\epsilon}\right)\left( \frac{2}{\frac{\epsilon}{4}} \right)^{k_{**}-1} \\
\leq ~~ & (3+\epsilon)\left(\frac{8}{\epsilon}\right)^{k_{**}} \\
\end{align*}
Note that a choice of $\epsilon_3$ closer to $\epsilon_2$ would slightly improve this bound, at the cost of increasing the 
threshold size of $n$ in (H3) of Step 5. 

\step{Step 5: Requirements for the lower bound on $n = |A|$}.
We collect the necessary hypotheses on the size of the graph:
\begin{enumerate}
 \item[(H1)] $n$ is large enough to allow for the choice of $m_*$ in the interval from Step 2 while preserving $m_{**} > k_{**}$,
$m_{**} > \frac{1}{\epsilon_2}$: 
\\ it suffices that $n > (k_{**}+1)q^{k_{**}-1}\left(\frac{3+\epsilon}{\epsilon}\right)$, 
which ensures $\frac{\epsilon n}{3+\epsilon} -q^{k_{**}-1} > k_{**}q^{k_{**}-1}$ 
\\ and also ensures that $n > 2q^{k_{**}-1}\left(\frac{3+\epsilon}{\epsilon}\right)$, for the calculation in Step 3
 \item[(H2)] $n$ is large enough for the sequence $\langle m_\ell : \ell < k_{**} \rangle$ 
 to satisfy $\operatorname{log}\operatorname{log}m_* \leq m_{**}$:
\\ it suffices that $n \geq (\operatorname{log}\operatorname{log} q^{k_{**}-1})(q^{k_{**}-1})\left(\frac{3+\epsilon}{\epsilon}\right)$
 \item[(H3)] $n$ is large enough for $m_{k_{**}-2}$ to satisfy condition (2)(e) of Corollary \ref{cor-k17}:
\\ it suffices that
$n \geq N(\frac{\epsilon}{3}, \frac{\epsilon}{2}, q^{k_{**}-1})\cdot\left(\frac{3+\epsilon}{\epsilon}\right)\cdot q^{k_{**}-2}$
where $N(\cdot,\cdot,\cdot)$ is from Claim \ref{k20} and incorporates the bounds from Fact \ref{vc-bound}. 
\end{enumerate}

Under these assumptions the sequence constructed in Step 2 will also satisfy conditions (2)(c),(e) of Corollary \ref{cor-k17}. 
By Step 0 and Step 2, all the hypotheses of that Corollary are satisfied.

\step{Step 6: Obtaining the partition}.
Assuming $n$ is sufficiently large, as described in Step 5, we have constructed a sequence $\langle m_\ell : \ell < k_{**} \rangle$
so that the graph $A$ and the constructed sequence satisfy the hypotheses of Corollary \ref{cor-k17}. Thus we obtain a partition of $A$ 
satisfying (1),(2),(3) of the Theorem. Condition (4) follows from Step 4, which completes the proof. 
\end{proof}

\begin{cor} \label{k26} Let $G$ be a graph with the non-$k_*$-order property.
For every $\epsilon \in (0, \frac{1}{2})$ there are $N, k$ as in Theorem \ref{m1} such that if $A \subseteq G$, 
$|A| \geq N$, then for some $m \leq k$, there is a partition $A = \langle A_i : i<m \rangle$ such that
each $A_i$ is $\epsilon$-excellent, and for every $0 \leq i < j < m$,

\begin{itemize}
\item $||A_i| - |A_j|| \leq 1$
\item $(A_i, A_j)$ is $\epsilon$-regular and
\item if $B_i \in [A_i]^{\geq \epsilon|A_i|}$
and $B_j \in [A_j]^{\geq \epsilon|A_j|}$ then 
\[ \left( \vrt d(B_i, B_j) < \epsilon \right) \lor \left( \vrt d(B_i, B_j) \geq 1-\epsilon \right)  \] 
\end{itemize}
\end{cor}

\begin{proof}
This is a slight weakening of Theorem \ref{m1}, which also replaces ``$\epsilon$-uniform,'' as defined in Claim \ref{nice}, by the more familiar $\epsilon$-regular via Claim \ref{k23}. For $\epsilon$-excellent, see Definition \ref{good}.
\end{proof}

\end{document}